\documentclass[a4paper,10pt]{article}
\usepackage[utf8]{inputenc}
\usepackage{amsmath}
\usepackage{tikz}
\usepackage{amsthm}
\usepackage{amssymb}
\usepackage{hyperref}
 \usepackage{cancel}
\usepackage{pdfsync}
\usepackage{mathrsfs}
\usepackage{systeme}
\usepackage{textcomp}
\usepackage{todonotes}
\usepackage{xcolor}
 \usepackage{cancel} 
\usepackage{caption2}     
\usepackage{float}        
\usepackage{fancyhdr}     
\usepackage{graphicx}     
\usepackage{lastpage}     
\graphicspath{{S-V paper/}} 

\newtheorem{lem}{Lemma}[section]

\newtheorem{thm}{Theorem}[section]

\numberwithin{equation}{section}

\theoremstyle{definition}
\newtheorem{defn}{Definition}[section]

\theoremstyle{remark}
\newtheorem{rmk}{Remark}[section]

\title{Boundary stabilization of flows in networks of open channels modeled by Saint-Venant equations}

\author{Amaury Hayat\thanks{CERMICS, \'{E}cole des Ponts ParisTech, Champs sur Marne, 77455 Marne la Vall\'{e}e, France (e-mail: amaury.hayat@enpc.fr).}	
\ and Yating Hu\thanks{Tongji University, 200092, Shanghai, China (e-mails:huyating@tongji.edu.cn, shang@tongji.edu.cn).} \footnotemark[1] 
\ and Peipei Shang \footnotemark[2]
}
\date{}

\begin{document}
	\maketitle
    \begin{abstract}
   This work investigates the boundary stabilization of flows in star-shaped and tree-shaped networks of open channels governed by the Saint-Venant equations with a friction term. Due to the existence of the friction term, the steady-states are non-uniform. We show that any such network can be stabilized with only controls at the terminal nodes of the network, even when there are no controls at the nodes inside the network. The number of control is optimal.
 The main tool we use is the Lyapunov approach, and the main challenge is that the state-of-the-art Lyapunov functions developed for Saint-Venant equations 
 with source terms cannot be used. In this work,
  we manage to construct a new efficient and explicit Lyapunov function and, in turn, we give explicit ranges of the control tuning parameters that depend only on the values of the given non-uniform steady-states at the ends of the branches. Moreover, this Lyapunov function also improves the existing conditions found in the last decade for a single channel modelled by Saint-Venant equations.
    \end{abstract}

\section{Introduction}

\subsection{Context}
Saint-Venant equations are prototypical one-dimensional hyperbolic systems that describe the flow of shallow water in open channels. These equations consist of the conservation of mass equations and  momentum equations, which are both extensively studied in mathematics and widely utilized in hydraulic engineering for rivers, canals, and other waterways (see \cite{graf1998fluvial}). The initial research on the exponential stability of the Saint-Venant equations using boundary feedback control starting in the 1980 and focused on a simplified single-channel model, which did not account for friction or slope. 
In the earliest work \cite{Li1984} by Greenberg and Li, the authors obtained the stabilization of $2\times 2$ system in the framework of $C^1$
solutions by characteristic method. This was later generalized in \cite{JCC,Qin,Zhao} and then in \cite{BC2015,C1,C122}. Later on, in the years 2000, Coron et al. identified in \cite{CBA2008,CAB2007} a general $H^2$-Lyapunov function 
which provided sufficient conditions for
boundary stabilization, still for the simplified model. 
This approach was improved
in \cite{bastin2011coron} to more general systems and, in 2017, Bastin and Coron introduced a new explicit Lyapunov function for ensuring the exponential stability of some physical $2\times 2$ hyperbolic systems with nonuniform steady-states including the Saint-Venant equations, this time with a friction term (but no slope) \cite{BC2017}. In 2019-2021, Hayat and Shang constructed yet another explicit quadratic Lyapunov function which could handle the general Saint-Venant equations (with arbitrary slope and friction terms) in \cite{HS} and generalized the results to density-velocity equations in \cite{hayat2021exponential}. A remarkable feature compared to the Lyapunov function of \cite{BC2017} is that it corresponds to the optimal Lyapunov function for a single channel in the sense of \cite{bastin2011coron}.
Another approach, the backstepping method, was also used to solve the stabilization problems,
see for example \cite{CK2013,Krstic2013,Krstic2015,hu2015control,hu2019boundary}. While very powerful, this approach often give rise to full-state feedback control that require an observer to be applied in real-life applications.
Recent studies have also investigated the stabilization of two-dimensional Saint-Venant equations
\cite{dia2013boundary,herty2024boundary,Yong2024}, a problem that also has significant practical value. \\

All the previous examples deal with single-channel flows, as a logical first step, while most applications real-world applications are inherently networked in nature. In the past two decades, networks of hyperbolic systems have been studied, for instance, in \cite{dick2014stabilization,gugat2023limits,perrollaz2014finite}.
In the last years, based on the practical applications, the controllability and stability of open channel networks modeled by the Saint-Venant equations have increasingly garnered the attention of researchers, with a particular focus on boundary controllability and boundary stabilization, given the physical nature of the problem. In \cite{LiNet20041,LiNet2005,LiNet20042}, Li et al. studied the exact controllability of Saint-Venant equations in various networks. We also refer to \cite{LWG} and the references therein for a comprehensive review of controllability of nodal profile for the networks without loop.
Recently, Li et al. have considered the exact boundary controllability of nodal profile for Saint-Venant systems on the networks with loops in \cite{LI20211,Lizhuang2019}.
For the stabilization problem, the current state of the art is still limited. When there is no source term, 
the exponential stabilization
is obtained in \cite{chba2002,JCC,LG2002,TAX2017,TAX2018}
for various networks.
When the source term arises, in \cite{BCA2009}, the authors considered a network in cascade where the slope and friction ``compensate" each other, thus resulted in constant equilibria.
A more recent work is \cite{hayat2023pi}, there the authors addressed the boundary Input-to-State Stability of a network
of channels in cascade, using proportional integral controllers and considering arbitrary cross sections, slope and friction. Nevertheless, a cascade system is the most simple example of network and is very far from covering networks in all generality.\\

In reality, there are numerous examples of river systems that exhibit star-shaped patterns.
When rivers flow into lakes or seas, sediment is deposited, forming deltas \cite{meade1996river}. Consequently, the rivers naturally bifurcate, creating a star-shaped pattern or a tree-shaped pattern. Typical examples of  deltas include the Fraser Delta and the Yellow River Delta and many delta regions are densely populated and serve as commercial and transportation hubs. The instability of the bifurcated water systems can lead to land erosion in deltas, resulting in significant economic losses (see \cite{blum2009drowning}).  In \cite{kleinhans2013splitting}, bifurcations are also found in alluvial fans, braided rivers, lowland rivers with meandering or anastomosing patterns. In these naturally formed river landforms, star-shaped structures can also be observed. According to \cite{syvitski2007morphodynamics,tang2022artificial}, artificial bifurcations have been constructed to prevent flooding and support agricultural irrigation. This highlights the importance of star-shaped and tree-shaped structures in hydraulic engineering. 
\textcolor{black}{
When the star-shaped network with source terms is concerned, one can refer to \cite{hayatHuShang2025} where the authors consider a supercritical channel and several subcritical channel that merge into a single supercritical channel and show that one only needs to have controls at the upstream of the incoming supercritical channel to stabilize the system. This situation, however, is very specific.}

\subsection{Contribution of this paper}

In this paper, we address the boundary stabilization of the Saint-Venant equations for \textcolor{black}{divergent} star-shaped and tree-shaped networks
taking into account the presence of the friction as a source term.  Especially, in this case, the steady-states are space-varying.\\

The difficulty in \textcolor{black}{such kind of} network is that, in real applications 
(such as irrigation channel for instance) 
applying a control at the internal node of the tree is, at best, challenging \textcolor{black}{when there are several outgoing branches
}. Rather, the boundary conditions at these nodes are imposed by the physics of the system. As a consequence, the relative number of available controls in a network can be much lower than in a single channel.
Our main contribution is to show that one can still achieve 
exponential stabilization in the $H^{2}$ norm of the whole 
tree-shaped networks with only one 
feedback control at each 
terminal nodes of the tree (see Figure \ref{fig:star} and \ref{f2}). This means, in particular, that we can surprisingly achieve the exponential stabilization of the full network with a number of controls that can be relatively small for large trees and no control at internal nodes of the network or at the root of the network.
Moreover, the conditions on the control tuning parameters are explicit and depend only on the values of the 
target
steady-states at the boundaries of the terminal branches.\\

To do so, we found a new Lyapunov function for the $H^{2}$-norm, different from the ones in \cite{BC2017,hayat2021exponential}
(which are not compatible with the physical conditions imposed by the network) while still being fully explicit. Just as the Lyapunov functions in \cite{BC2017,hayat2021exponential} this Lyapunov function has the remarkable property of being suitable for any length of the channels, while this usually does not happen for hyperbolic systems in general (see \cite[Section 2.7]{BCBook} or \cite{C1,C122}). As a side product, it also improves the set of best known conditions for the single channel geometry by providing explicit sufficient conditions that are not contained in the ones of \cite{BC2017,hayat2021exponential} (see Theorem \ref{thm:single}).
\\

On the other hand, since the Saint-Venant equations are a typical density-velocity system--a class of system omnipresent in physics which encompass all systems that consist in a flux conservation and a momentum equation \cite{BC2017,hayat2021exponential}. It is expected that our approach could be extended to other density-velocity systems in networks. Among which, the most studied examples are the isentropic Euler equations
for the gas motion in pipelines, for which
one can refer to the related works by Gugat et al
\cite{GDL2011,gugat2017coupling} and the references therein. \textcolor{black}{In particular \cite{gugat2017coupling} provides a study of a general junction in a network of isothermal Euler equations without source term.}\\

The organization of the paper is the following: we state the main results in Section \ref{sec:mainresults}, i.e. the stabilization result of star-shaped model in Section \ref{star} and tree-shaped network in Section \ref{tree}.
We show the result in the case of a linearized system and introduce the main tools in Section \ref{sec:lyap} which contains the main proof. 
In Section \ref{sec:nonlinear}, we show that the results of Section \ref{sec:lyap} can be extended to the nonlinear system.
In Section \ref{LtreeL}, we generalize the result to the tree-shaped network and give the sketch of the proof for the linearized system.
Finally, we give conclusion in Section \ref{con} and some useful computations can be found in the Appendix.

	\section{Main Results}
    \label{sec:mainresults}
We start by presenting the main result for a star-shaped network with a single internal node in Section \ref{star}, and we state the main result in full generality in Section \ref{ssec:treeshape}.
    \subsection{The Star-shaped Model}\label{star}
	In this section, we  consider a star-shaped model  of divergent flow composed of a main channel and $n-1$ ($n\geq 3$) branches, the direction of the water is shown in Figure \ref{fig:star}. Such a geometry can be seen as the basic brick of a tree-shaped network.
	\begin{figure}[h]
		\renewcommand{\captionfont}{\small}
		\centering
		\includegraphics[height=5cm]{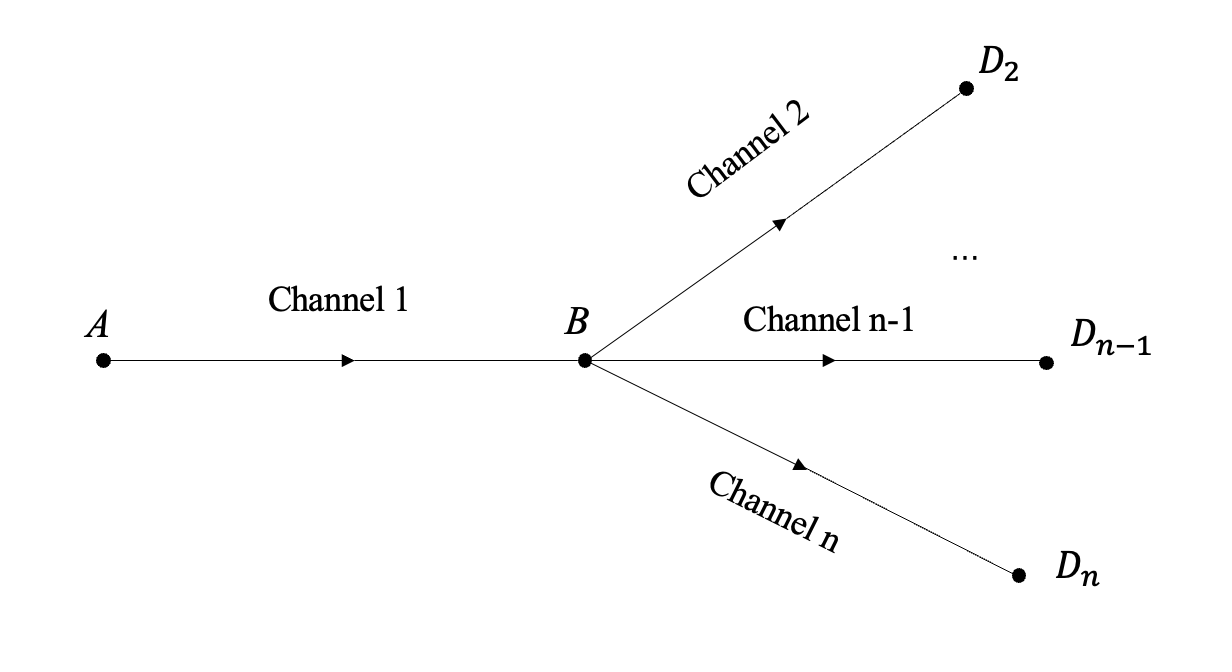}
		\caption{\textbf{Star-shaped model of divergent flow.} \label{fig:star}}
	\end{figure}
	Consider the Saint-Venant equations  with  different friction while maintaining horizontal orientation of each channels defined on $ [0,+\infty )\times[0, L_i]$ (here and hereafter, unless specified, we always assume that $i\in \{1,2,\cdots, n\}$) as follows
	
	\begin{equation}\label{sys01}
		\begin{aligned}
			&	\partial_{t} H_{i}+\partial_{x}(H_iV_i)=0,\\
			& \partial_{t}V_i +V_i\partial_{x}V_i+g\partial_{x}H_i+g\frac{C_iV_i^2}{\textcolor{black}{H_i^{p}}}=0,
		\end{aligned}
	\end{equation}

 where $L_i$ is the length of the channel $i$,
$H_i=H_i(t,x)$ is the height of the water,  $V_i=V_i(t,x)$ is the horizontal water velocity,  $C_i $ is the friction coefficient, \textcolor{black}{$p\geq 0$ is a parameter of the friction model,} $g$ is the gravitational acceleration. The model has $n$ simple nodes $A$, $D_2$, $\cdots$, $D_n$
 and one multiple node $B$ which serves as the junction between the main channel and the $n-1$ branches. To better approximate real-world scenarios and enhance the applicability when considering boundary conditions, we consider a constant imposed inflow at the starting node $A$ as in \cite{hayat2021exponential} and refrain from implementing any control at multiple node $B$, as controlling at multiple node $B$ is challenging in practical situations.
	At  nodes $A$, $B$, $D_j$ (here and hereafter, unless specified, we always assume that $j\in \{2,3, \cdots, n\}$), the boundary conditions are given as
	\begin{equation}
		\begin{aligned}\label{bou01}
			&	A: H_1(t,0)V_1(t,0)=Q_1 ,\\
			& B:  H_1(t,L_1)V_1(t,L_1)=\sum_{j=2}^{n} H_j(t,0)V_j(t,0),\\
			& \ \ \ \  \ H_1(t,L_1)=H_j(t,0),\\
			& D_j:V_j(t,L_j)=\mathscr{B}_j(H_j(t,L_j)),\\
		\end{aligned}
	\end{equation}
	where the control functions $\mathscr{B}_j$: $\mathbb{R} \rightarrow \mathbb{R}$ are of class $C^2$, and the flux $Q_1$ is a positive constant. The boundary conditions at multiple node $B$ are natural conditions,  which are respectively the conservation of mass and the continuity of the water pressure \cite{Herty2013}
    or water level \cite{LiNet2005}.\\

{\color{black}
\begin{rmk}[friction model]
Different choice of $p$ corresponds to different friction models. Our main theorems, Theorems \ref{thm0}-\ref{thm:single}, hold for any $p\geq 0$, in particular this framework covers two widely used models:
\begin{itemize}
    \item when $p=1$ this corresponds to the Chézy model \cite{BCBook,Chanson,Chaudhry} (see also \cite{bastin2011coron,BCA2009}). This model is sometimes written as $kV|V|/H$ where $k=gC$ is an adimensional constant \cite{BCBook}.
    \item when $p=4/3$ this corresponds to the Manning-Strickler model \cite[Equation (2.1)]{chertock2015well} (see also \cite{GugatLeugering2009}).
\end{itemize}
\end{rmk}
}
    
	For the system \eqref{sys01}, the steady-states $ (H_i^*(x),V_i^*(x))$ satisfy
	\begin{equation}
    \label{eq:steady0}
		\begin{aligned}
			&(H_i^* V_i^*)_x=0,\\
			&\Big(\frac{V_i^{*2}}{2}+gH_i^*  \Big)_x+g\frac{C_i V_i^{*2}}{\textcolor{black}{H_i^{*p}}}=0.
		\end{aligned}
	\end{equation}
	The first equation can be rewritten as
\begin{equation}\label{nve}
		\begin{aligned}
			H_i^* (V_i^*)_x+(H_i^*)_x V_i^*=0,
		\end{aligned}
	\end{equation}
In this paper, we consider the physical steady-states: $H^*_i (x) > 0$,  $V^*_i (x) > 0$
 which are in the subcritical case for all $x \in [0, L_i]$, i.e. the following conditions hold
	\begin{equation}\label{subcritical}
		 	gH_i^*(x)-V_i^{*2}(x)>0.
	\end{equation}
Under condition \eqref{subcritical} and by \eqref{nve}, \eqref{eq:steady0} means that
    \begin{equation}\label{steady-state}
		\begin{aligned}
			&	( H_i^*)_x=-\frac{gC_i V_i^{*2}}{\textcolor{black}{(H_i^{*})^{p-1}}(gH_i^*-V_i^{*2})},\\
			&( V_i^*)_x=\frac{gC_i V_i^{*3}}{\textcolor{black}{(H_i^{*})^{p}}(gH_i^*-V_i^{*2})}.
		\end{aligned}
	\end{equation}
Moreover,  from \eqref{bou01}
\begin{equation}
\begin{aligned}\label{bou01s}
			&	A: H_1^{*}(0)V_1^{*}(0)=Q_1 ,\\
			& B:  H_1^{*}(L_1)V_1^{*}(L_1)=\sum_{j=2}^{n} H_j^{*}(0)V_j^{*}(0),\\
			& \ \ \ \  \ H_1^{*}(L_1)=H_j^{*}(0),\\
			& D_j:V_j^{*}(L_j)=\mathscr{B}_j(H_j^{*}(L_j)).\\
		\end{aligned}
\end{equation}

 From equations \eqref{steady-state} and the subcritical condition \eqref{subcritical}, it can be observed that $H_i^*(x)$ is a monotonically decreasing function and $V_i^*(x)$ is a monotonically increasing function. In fact, it can be shown that, if $L_{i}$ is sufficiently large, these steady-states cease to exist \textcolor{black}{as classical solutions}. More precisely, there exists a series of critical points $x^{i}_0$ where
the solution to equations \eqref{steady-state} will blow up (see \cite{GugatLeugering2009,C122} for more details). For any given flux $Q_i>0$ such that $H^*_i(x)V^*_i(x)=Q_i$, looking at \eqref{steady-state},
$x^{i}_0$ should satisfy 
\begin{equation}\label{x0}
    g H_i^*(x^{i}_{0}) = V_i^{*2}(x^{i}_{0}) = \frac{Q_i^{2}}{H_i^{*2}(x^{i}_{0})},
\end{equation}
it follows that
\begin{equation}\label{dcr}
H_i^*(x_{0}^i) = \left(\frac{Q_i}{\sqrt{g}}\right)^{\frac{2}{3}}.
\end{equation}
Thus, in the following we always assume that the lengths $L_i$ are such that the steady-states exist on $[0,L_i]$, i.e. $L_{i}<x_{0}^{i}$.

 We first give the definition of  local exponential stability of the steady-states in the $H^2$-norm. 
\begin{defn}
    The steady-state $(H_i^*(x),V_i^*(x))$ of the system \eqref{sys01}, \eqref{bou01} 
    is (locally) exponentially stable for the $H^2$-norm if there exist $\delta>0$, $\nu>0$ and $C>0$ such that, for any initial data $(H_i^0(x),V_i^0(x))\in H^2((0,L_i);\mathbb{R}^2)$ satisfying
    \begin{equation}
  \|(H_i^0-H_i^*,V_i^0-V_i^*)\|_{H^2((0,L_i);\mathbb{R}^2)}<\delta
    \end{equation}
    and the first order compatibility conditions associated to the system \eqref{sys01} (see \cite[Section 4.5.2]{BCBook}), there exists a unique solution $(H_i,V_i)\in C^{0}([0,+\infty);H^2((0,L_i);\mathbb{R}^2))$ to the Cauchy problem \eqref{sys01}, \eqref{bou01} 
    with
    $H_i(0,\cdot)=H_i^0$, $V_i(0,\cdot)=V_i^0$
    and, in addition, it satisfies
    \begin{align}
\sum_{i=1}^{n}&\|(H_i(t,\cdot)-H_i^*,V_i(t,\cdot)-V_i^*)\|_{H^2((0,L_i);\mathbb{R}^2)}   \leq Ce^{-\nu t}\sum_{i=1}^{n}\|(H_i^0-H_i^*,V_i^0-V_i^* )\|_{H^2((0,L_i);\mathbb{R}^2)}.
    \end{align}
\end{defn}
One of our two main theorems is the following
\begin{thm}\label{thm0}
   For any boundary controls $\mathscr{B}_j$, $j\in\{2,\cdots,n\}$ satisfying 
   \begin{equation}\label{paraphy}
      \mathscr{B}'_j(H_j^*(L_j)) \in \mathbb{R}\backslash \left[ -\sqrt{\frac{g}{H^*_j(L_j)}}\frac{\lambda_j^++m_j\lambda_j^-}{\lambda_j^+-m_j\lambda_j^-}, -\sqrt{\frac{g}{H^*_j(L_j)}}\frac{\lambda_j^+-m_j\lambda_j^-}{\lambda_j^++m_j\lambda_j^-}\right],
   \end{equation}
where 
\begin{align}
\label{use001}
   \lambda_j^{\pm}=&\sqrt{gH^*_j(L_j)}\pm\frac{Q_j}{H^*_j(L_j)},\nonumber\\
    m_j=&\frac{\frac{1}{2}(H_j^{*}(L_j))^{\frac{\textcolor{black}{3+2p}}{2}}+\frac{\sqrt{g}}{\textcolor{black}{(3+p)}Q_i}(H_j^{*}(L_j))^{\textcolor{black}{3+p}}+\frac{\textcolor{black}{(1+p)}\sqrt{g}}{\textcolor{black}{2(3+p)}Q_i}(H_j^{*}(0))^{\textcolor{black}{3+p}}+\frac{Q_i}{2\sqrt{g}}((H^*_j(L_j))^{\textcolor{black}{p}}-(H^*_j(0))^{\textcolor{black}{p}})}{-\frac{1}{2}(H_j^{*}(L_j))^{\frac{\textcolor{black}{3+2p}}{2}}+\frac{\sqrt{g}}{\textcolor{black}{(3+p)}Q_i}(H_j^{*}(L_j))^{\textcolor{black}{3+p}}+\frac{\textcolor{black}{(1+p)}\sqrt{g}}{\textcolor{black}{2(3+p)}Q_i}(H_j^{*}(0))^{\textcolor{black}{3+p}}+\frac{Q_i}{2\sqrt{g}}((H^*_j(L_j))^{\textcolor{black}{p}}-(H^*_j(0))^{\textcolor{black}{p}})}
\end{align}
are all constants depending only on the values of $H^*_j(x)$ at two ends $x=0$ and $x=L_j$.
The (steady-state $(H^{*}_{i},V^{*}_{i})$ of the) nonlinear hyperbolic system \eqref{sys01}, \eqref{bou01} 
   is exponentially stable for the $H^2$-norm .
\end{thm}

\begin{rmk}
This theorem and its counterpart for the tree-shaped network (Theorem \ref{thm0tree}) are surprising for two main reasons:
\begin{itemize}
\item It shows that there is no need to have a control at the junction. This is particularly striking for a full tree-shaped network (see Theorem \ref{thm0tree}) where there is no need to have controls inside the network to stabilize the full network: one only needs to have controls at the ends of the network. 
\item The (sufficient) boundary conditions for the stabilization are perfectly explicit and they only \textcolor{black}{depend} on the target height $H^{*}_j$ of the branches at $x=0$ and $x=L_j$. This was already the case for a single channel in \cite{hayat2021exponential}, however, there is still no simple explanation today as of why. \textcolor{black}{Note that, in \cite{hayat2021exponential}, the sufficient stability condition 
\textcolor{black}{for a single channel defined on $[0,L]$ when the system has a single control at $x=L$}
is
\begin{equation}
\mathcal{B}'(H^{*}(L)) \in \mathbb{R}\setminus \left[-\frac{g}{V^{*}(L)},-\frac{H^{*}(L)}{V^{*}(L)}\right],
\end{equation}
but} the Lyapunov function in \cite{hayat2021exponential} cannot be used in this case, precisely because of the boundary conditions imposed by the junction.
\end{itemize}    
\end{rmk}

\begin{rmk}\label{rem}
 In the special case where $V_{j_{0}}^{*}=0$ for some $j_{0}\in\{2,\cdots,n\}$, the result still holds with the condition on $\mathscr{B}'_{j_{0}}(H_{j_{0}}^*(L_{j_{0}}))$ being replaced by
\begin{equation}
      \mathscr{B}'_{j_{0}}(H_{j_{0}}^*(L_{j_{0}})) \in (0,+\infty).
   \end{equation}
\end{rmk}

\begin{rmk}[Number of controls]
One can note that this is a $2n\times 2n$ system with $n-1$ boundary controls. This number of control is optimal in the following sense: there are $n-1$ components that are propagating from $D_{j}$ to $B$. In the particular case where there is no friction, for any $j\in\{2,\cdots,n\}$ there is no source term for Channel $j$ and the component propagating in the direction $D_{j}$ to $B$ is not coupled to any other component except at the node $B$. Hence, its behavior cannot be influenced if there is no boundary control in $D_{j}$. \textcolor{black}{One can also look at \cite[Remark 3]{gugat2010stars} where the optimality of $n-1$ control was already observed in the case of a star-shaped system.}
\end{rmk}

\subsection{Tree-shaped Network}\label{tree}
\label{ssec:treeshape}
In this section, we consider a more general tree-shaped network. We intend to show that, in this framework as well, the full network can be controlled using only controls at the terminal simple nodes. We consider a network composed of $n$ channels indexed by $i\in\mathcal{I}=\{1,\cdots, n\}$ and with length $L_{i}$. 
Let $A$ denote the root node, that is the starting point of the network and label the trunk channel as $i=1$.
We assume that the network is  tree-shaped: all multiple nodes representing an internal junction in the network have
only one single incoming channel but possibly several outgoing channels (see Figure \ref{f2} for the illustration of a typical graph). 
 \begin{figure}[h!]
		\renewcommand{\captionfont}{\small}
		\centering
		\includegraphics[height=7cm]{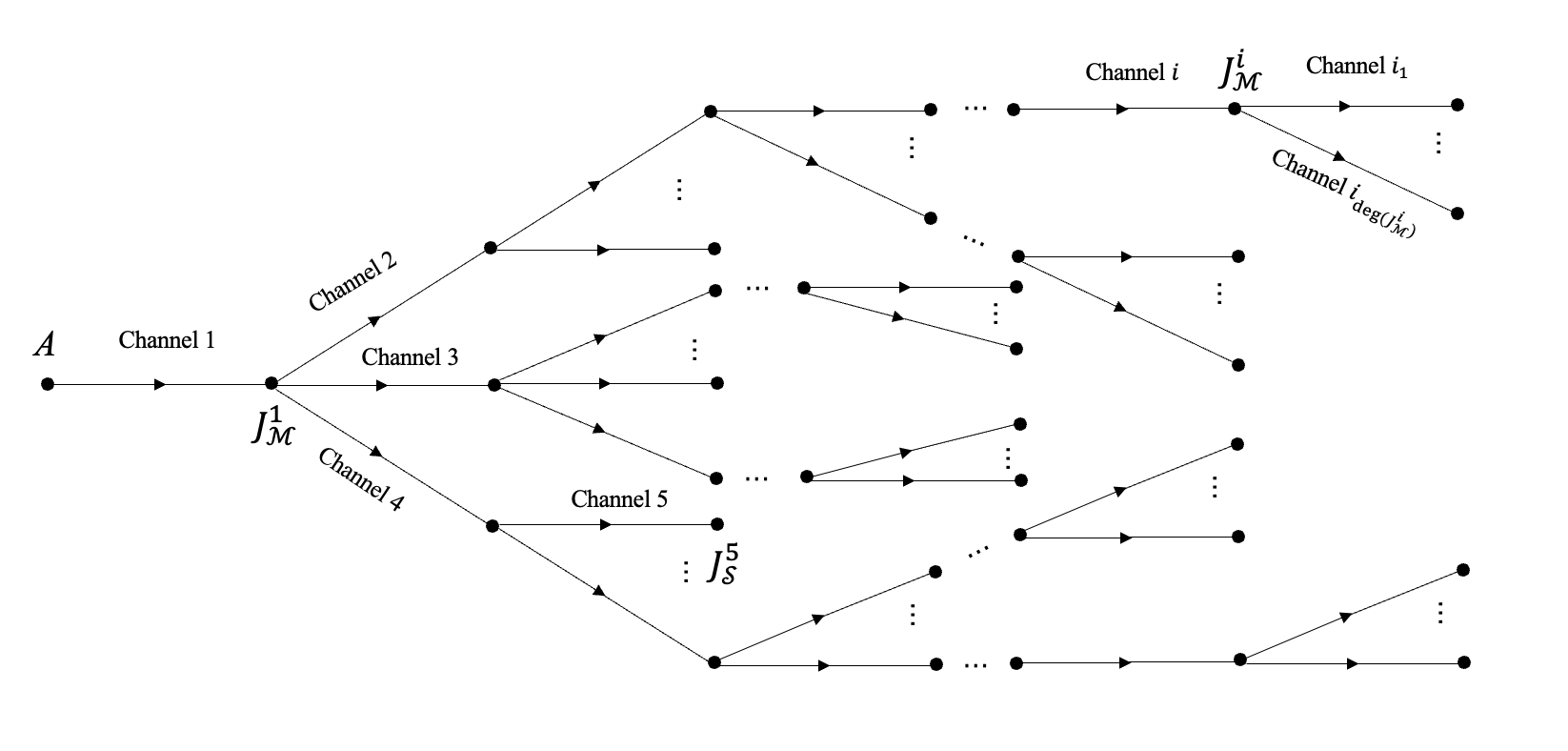}
		\caption{\textbf{Example diagram of a tree-shaped channels network model}. In this example, $2,3,4 \in \mathcal{I}_{out}^{1}$, $\text{deg}(J_{M}^{1})=4$, and $1$, $i \in \mathcal{M}$ and $5\in \mathcal{S}$.}\label{f2}
\end{figure}
We adapt the notation in \cite{Lizhuang2019} and denote these multiple nodes by $J^i_M$ where $i$ is the index of the incoming channel (i.e. the channel that ends at this node), and we denote
by
$\mathcal{I}^i_{out}$ the set of indices corresponding to outgoing channels at this multiple node (i.e. the channels connected with the incoming $i$-th channel). 
The total number of  channels connected with the multiple node $J^i_M$
is called the degree of this node, denoted by $\deg(J^i_M)$.
Correspondingly, we denote the simple nodes at the ends of the network (i.e. the terminal nodes of the network) by $J^i_S$ 
where $i$ is the index of the incoming channel (i.e. the channel that ends at this node). We also denote by $\mathcal{M}$ the set of indices of channels that end at a multiple node and by $\mathcal{S}$ the set of channels that end at a simple node.
The boundary conditions are given as follows
\begin{equation}
		\begin{aligned}\label{treebound}
			A:\	&H_1(t,0)V_1(t,0)=Q_1 ,\\
		  J^i_M:\ &  H_i(t,L_i)V_i(t,L_i)=\sum_{j\in \mathcal{I}^i_{out}} H_j(t,0)V_j(t,0),\quad i\in\mathcal{M},\\
			&  H_i(t,L_i)=H_j(t,0),\quad i\in\mathcal{M},\quad j\in \mathcal{I}^i_{out},\\
			J^i_S:\ & V_i(t,L_i)=\mathscr{B}_i(H_i(t,L_i)),\quad i\in\mathcal{S},\\
		\end{aligned}
	\end{equation}
where, as previously, the control functions $\mathscr{B}_i$: $\mathbb{R} \rightarrow \mathbb{R}$ are of class $C^2$
and are only applied to the end of the branches, the flux $Q_{1}$ is a positive constant and the boundary conditions at the multiple node $J_{M}^{i}$ are again the conservation of mass and the continuity of the water pressure or water level.

We emphasize that in contrast to the star-shaped model, where $H_{1}(t,0)V_{1}(t,0)$ is a constant, the flux $H_{j}(t,0)V_{j}(t,0)$ $(j\in\mathcal{I}^i_{out},i\in\mathcal{M})$ (i.e. the beginning of each outgoing channel that is connected with the $i$-th incoming channel)
is not necessarily constant 
but depends on 
other channels
and is subject to the natural physical condition, i.e. the conservation of mass and the continuity of the water pressure or water level (see \eqref{treebound}). Therefore the previous analysis for the star-shaped model does not directly transfer to this case. However, and in particular thanks to Lemma \ref{lem:cj} below, we can still show the following result

\begin{thm}\label{thm0tree}
   For any boundary controls $\mathscr{B}_i$ $(i\in\mathcal{S})$ satisfying 
      \begin{equation}\label{paraphyi}
      \mathscr{B}'_i(H_i^*(L_i)) \in \mathbb{R}\backslash \left[ -\sqrt{\frac{g}{H^*_i(L_i)}}\frac{\lambda_i^++m_i\lambda_i^-}{\lambda_i^+-m_i\lambda_i^-}, -\sqrt{\frac{g}{H^*_i\textcolor{black}{(}L_i)}}\frac{\lambda_i^+-m_i\lambda_i^-}{\lambda_i^++m_i\lambda_i^-}\right],
   \end{equation}
  where $\lambda_i^{\pm}$ and $m_i$ are defined by \eqref{use001} with the index $j$ being replaced by $i$,
   the (steady-state $(H^{*}_{i},V^{*}_{i})$ of the) nonlinear hyperbolic system \eqref{sys01}, 
   \eqref{treebound} is exponentially stable for the $H^2$-norm .
\end{thm}
Moreover, Remark \ref{rem} still holds for the tree-shaped network.

\subsection{A side result for a single channel}
Theorems \ref{thm0} and \ref{thm0tree} are obtained thanks to an analysis of the network and a Lyapunov approach. As it turns out, the new Lyapunov function we derive can also improve the best known conditions for the stabilization of a single channel currently given in \cite{BC2017,hayat2021exponential} by providing sufficient conditions that are not included in the sufficient conditions of \cite{BC2017,hayat2021exponential}. Indeed, consider the system
	\begin{equation}\label{sys01one}
		\begin{aligned}
			&	\partial_{t} H+\partial_{x}(HV)=0,\\
			& \partial_{t}V +V\partial_{x}V+g\partial_{x}H+g\frac{CV^2}{H}=0,
		\end{aligned}
	\end{equation}
    on $[0,L]$ with boundary conditions
    \begin{equation}
    \label{eq:boundone}
        V(t,0) = \mathcal{G}_{1}(H(t,0)),\;\; V(t,L) = \mathcal{G}_{2}(H(t,L)),
    \end{equation}
     where the control functions $\mathcal{G}_{1}$ and $\mathcal{G}_{2}$: $\mathbb{R} \rightarrow \mathbb{R}$ are of class $C^2$ and consider a steady-state $(H^{*}(x),V^{*}(x))$ with $gH^{*}(x)-V^{*2}(x)>0$. We have the following theorem
\begin{thm}
\label{thm:single}
For any boundary controls $\mathcal{G}_{1}$ and $\mathcal{G}_{2}$ such that
    \begin{equation}
    \label{eq:condsingle}
        \mathcal{G}_{1}'(H^{*}(0))\in(-\infty,0],\;\;\mathcal{G}_{2}'(H^{*}(L))\in \mathbb{R}\backslash \left[ -\sqrt{\frac{g}{H^*(L)}}\frac{\lambda^++m\lambda^-}{\lambda^+-m\lambda^-}, -\sqrt{\frac{g}{H^*(L)}}\frac{\lambda^+-m\lambda^-}{\lambda^++m\lambda^-}\right],
    \end{equation}
where $\lambda^{\pm}$ and $m$ is given by \eqref{use001} with $H^{*}$ instead of $H_{j}^{*}$, the (steady-state $(H^*,V^*)$ of the) nonlinear system \eqref{sys01one}, \eqref{eq:boundone} is exponentially stable for the $H^{2}$-norm.
\end{thm}
\begin{rmk}
    Note that the condition on the left $\mathcal{G}_{1}'(H^{*}(0))\in(-\infty,0]$ is very permissive and much weaker than the conditions known so far, \textcolor{black}{that are
    \begin{itemize}
        \item $\mathcal{G}_{1}'(H^{*}(0))\in [-\frac{g}{V^{*}(0)},-\frac{V^{*}(0)}{H^{*}(0)}]$ derived in \cite{hayat2021exponential};
        \item $\mathcal{G}_{1}'(H^{*}(0))\in [-\frac{g}{V^{*}(0)}-\sqrt{\frac{g^{2}}{V^{*}(0)^{2}}-\frac{g}{H^{*}(0)}},-\frac{g}{V^{*}(0)}+\sqrt{\frac{g^{2}}{V^{*}(0)^{2}}-\frac{g}{H^{*}(0)}}]$ derived in \cite{BC2017}.
    \end{itemize}
    Note that all these conditions coincide in the limit $V^{*}(0)\rightarrow 0$.
    }
\end{rmk}
The proof of Theorem \ref{thm:single} is shown in Appendix \ref{app:single}.

\section{A Lyapunov Function for the Linearized  System of the Star-shaped Model}
\label{sec:lyap}
 To linearize the system \eqref{sys01}--\eqref{bou01}, we introduce the disturbances $\left(h_i(t,x), v_i(t,x)\right)$ as
 \begin{equation}\label{disturbance}
		h_i(t,x)=H_i(t,x)-H_i^*(x), \ \  v_i(t,x)=V_i(t,x)-V_i^*(x),
	\end{equation}
 then the  linearized system of \eqref{sys01}  is
	\begin{equation}\label{sysl01}
		\begin{aligned}
			&\partial_{t}h_i +V_i^*\partial_{x}h_i +H_i^*\partial_{x}v_i +(\partial_{x}V_i^*)h_i +(\partial_{x}H_i^*)v_i=0,\\
			&\partial_{t}v_i+g\partial_{x}h_i+V_i^*\partial_{x}v_i-gC_i\frac{V_i^{*2}}{H_i^{*2}}h_i+\left( \partial_{x}V_i^*+2gC_i\frac{V_i^*}{H_i^*}\right)v_i=0,
		\end{aligned}
	\end{equation}
	and the boundary conditions \eqref{bou01} become
\begin{equation}
	\begin{aligned}\label{boul01}
			&	A: h_1(t,0)=-\frac{H_1^*(0)}{V_1^*(0)}v_1(t,0),\\
			& B: 
   v_1(t,L_1)=\sum_{j=2}^{n}v_j(t,0),\\
			& \ \ \ \  \ h_1(t,L_1)=h_j(t,0),\\
			& D_j:v_j(t,L_j)=k_jh_j(t,L_j),\\
		\end{aligned}
	\end{equation}
	here, $k_j=\mathcal{B}_j'(H^*_j(L_j))$ 
 represent the tuning parameters to be adjusted. For any initial condition $(h_i^0(x),v_i^0(x)) \in L^2((0,L_i);\mathbb{R}^2)$, the Cauchy problem \eqref{sysl01}--\eqref{boul01} with
	\begin{equation}\label{ini}
		h_i(0,x)=h_i^0(x), \ \ v_i(0,x)=v_i^0(x)
	\end{equation}
 is well-posed and has a unique solution (in $C^{0}([0,+\infty);\prod\limits_{i=1}^{n}L^{2}((0,L_{i}),\mathbb{R}^{2}))$) (see \cite[Appendix A-B]{BCBook}). Moreover, we have

	\begin{thm}\label{thm1}
  If the control parameters $k_j$, $j\in\{2,3,\cdots,n\}$ satisfy the condition
\begin{equation}
  \label{eq:hypki}
      k_j \in \mathbb{R}\backslash \left[ -\sqrt{\frac{g}{H^*_j(L_j)}}\frac{\varphi_j(L_j)+\bar\eta_j(L_j)}{\varphi_j(L_j)-\bar\eta_j(L_j)}, -\sqrt{\frac{g}{H^*_j(L_j)}}\frac{\varphi_j(L_j)-\bar\eta_j(L_j)}{\varphi_j(L_j)+\bar\eta_j(L_j)}\right],
   \end{equation}
  where $\varphi_j$ and $\bar\eta_j$ are defined explicitly by \eqref{varphii} and \eqref{exbareta} below, then the system \eqref{sysl01}--\eqref{boul01} is exponentially stable for the $L^2$-norm.
	\end{thm}
     To prove Theorem \ref{thm1}, we further simplify the linearized system \eqref{sysl01} by introducing Riemann invariants 
	\begin{equation}\label{RiemannIn}
		\begin{aligned}
			&y_{1i}(t,x)=v_i(t,x)+h_i(t,x)\sqrt{\frac{g}{H_i^*(x)}}, \\
			& y_{2i}(t,x) =v_i(t,x)-h_i(t,x)\sqrt{\frac{g}{H_i^*(x)}}.
		\end{aligned}
	\end{equation}
		Then system \eqref{sysl01} can be written in the characteristic form as follows
		\begin{equation}\label{chaform}
			\begin{aligned}
				&\partial_{t}y_{1i}+\lambda_{1i}(x)\partial_{x}y_{1i}+\gamma_{1i}(x)y_{1i}+\delta_{1i}(x)y_{2i}=0,\\
				&\partial_{t}y_{2i}-\lambda_{2i}(x)\partial_{x}y_{2i}+\gamma_{2i}(x)y_{1i}+\delta_{2i}(x)y_{2i}=0,
			\end{aligned}
		\end{equation}
		where
		\begin{equation}\label{lam1}
			\lambda_{1i}(x)=V_i^*(x)+\sqrt{gH_i^*(x)}>  \lambda_{2i}(x)= -V_i^*(x)+\sqrt{gH_i^*(x)}>0,
		\end{equation}
		and
		\begin{equation}\label{all>0}
			\begin{aligned}
				&	\gamma_{1i}(x)=g\frac{C_iV_i^{*2}}{\textcolor{black}{H_i^{*p}}}\left[-\frac{3}{4(\sqrt{gH_i^*}+V_i^*)}+\frac{1}{V_i^*}-\frac{\textcolor{black}{p}}{2\sqrt{gH_i^*}}\right],\\
				&\delta_{1i}(x)=g\frac{C_iV_i^{*2}}{\textcolor{black}{H_i^{*p}}}\left[-\frac{1}{4(\sqrt{gH_i^*}+V_i^*)}+\frac{1}{V_i^*}+\frac{\textcolor{black}{p}}{2\sqrt{gH_i^*}}\right],\\
				&\gamma_{2i}(x)=g\frac{C_iV_i^{*2}}{\textcolor{black}{H_i^{*p}}}\left[\frac{1}{4(\sqrt{gH_i^*}-V_i^*)}+\frac{1}{V_i^*}-\frac{\textcolor{black}{p}}{2\sqrt{gH_i^*}}\right],\\
				&\delta_{2i}(x)=g\frac{C_iV_i^{*2}}{\textcolor{black}{H_i^{*p}}}\left[\frac{3}{4(\sqrt{gH_i^*}-V_i^*)}+\frac{1}{V_i^*}+\frac{\textcolor{black}{p}}{2\sqrt{gH_i^*}}\right].\\
			\end{aligned}
		\end{equation}
		These expressions can also be found for instance in \cite{bastin2011coron}. Simple calculations using the subcritical condition \eqref{subcritical} lead to
        \begin{equation}\label{positive}
        \gamma_{1i}>0,\, \delta_{1i}>0,\,\gamma_{2i}>0\, \text{ and}\ \delta_{{2i}}>0.
        \end{equation}
        In the following computations, for the sake of simplicity, we still keep the expression of the physical boundary conditions for  nodes $A$ and $B$ as in \eqref{boul01} rather than re-expressing them in the new variables in \eqref{RiemannIn}.
        For  nodes $D_j$, the boundary conditions become
		\begin{equation}\label{nbou01}
			\begin{aligned}
				&D_j: y_{1j}(t,L_j)=c_j y_{2j}(t,L_j), \\
			\end{aligned}
		\end{equation}
		where
		\begin{equation}\label{eqcj}
	c_j=\frac{1+k_j\sqrt{\frac{H_j^*(L_j)}{g}}}{k_j\sqrt{\frac{H_j^*(L_j)}{g}}-1}.  \\	
		\end{equation}
In order to prove Theorem \ref{thm1}, we first introduce a useful lemma, the proof of which can be \textcolor{black}{easily checked using the result given in \cite{hayat2021exponential} as long as the parameter $p$ is nonnegative}.
\begin{lem}\label{etaresult}
For $i\in\{1,2,\cdots,n\}$, denote by
\begin{equation}\label{eta0i}
 \eta_{0i}(x)=\frac{\lambda_{2i}(x)}{\lambda_{1i}(x)}\varphi_i(x),
\end{equation}
then for any $x\in[0, L_i]$,
\begin{equation}
    \eta_{0i}'=\left|\frac{\delta_{1i}\varphi_i}{\lambda_{1i}}+\frac{\gamma_{2i}}{\lambda_{2i}\varphi_i}\eta^2_{0i}\right|,
\end{equation}
where
\begin{equation}\label{varphii}\varphi_i(x)=\frac{\varphi_{1i}(x)}{\varphi_{2i}(x)}\end{equation}
with
\begin{equation}\label{varphi12}
\varphi_{1i}(x)=e^{\int_{0}^{x}\frac{\gamma_{1i}(s)}{\lambda_{1i}(s)}ds},\quad \varphi_{2i}(x)=e^{-\int_{0}^{x}\frac{\delta_{2i}(s)}{\lambda_{2i}(s)}ds}.
\end{equation}
\end{lem}
Next, we prove the following key lemma
\begin{lem}\label{lem1}
For $i\in\{1,2,\cdots,n\}$, let $x^{i}_0$ be the critical point defined by \eqref{x0},
then
\begin{equation}\label{intergaline}
				\int_{0}^{x}e^{\int_{0}^{s}2\frac{\gamma_{2i}(\xi)}{\lambda_{1i}(\xi)}d\xi}\frac{\gamma_{2i}(s)}{\lambda_{2i}(s) \varphi_i(s)}ds<\frac{\lambda_{1i}(0)}{\lambda_{1i}(0)-\lambda_{2i}(0)}
			\end{equation}
   holds for any
$x \in [0, x^{i}_0)$. 
\end{lem}
\begin{proof}
In order to simplify the left-hand side of the inequality \eqref{intergaline}, we use the dynamics \eqref{steady-state} of the steady-state. 
 Firstly, from \eqref{steady-state}, we 
 have
\begin{equation}
\label{exprCi}
     gC_iV_i^{*2}=(V_i^{*2}-gH^*_i)\textcolor{black}{H^{p-1}}(H_i^*)_x=-\lambda_1\lambda_2\textcolor{black}{H^{p-1}}(H_i^*)_x.
\end{equation}
Substituting \eqref{exprCi} into \eqref{all>0}, we can rewrite $\gamma_{1i}$,$\gamma_{2i}$, $\delta_{1i}$ and $\delta_{2i}$ in the form that depends on the dynamics of $H^*_i$ as follows
\begin{equation}\label{mall>0}
			\begin{aligned}
				&	\gamma_{1i}(x)=-\frac{(H_i^*)_x}{H_i^*}\lambda_{1i}\lambda_{2i}\left[-\frac{3}{4\lambda_{1i}}+\frac{1}{V_i^*}-\frac{\textcolor{black}{p}}{2\sqrt{gH_i^*}}\right],\\
				&\delta_{1i}(x)=-\frac{(H_i^*)_x}{H_i^*}\lambda_{1i}\lambda_{2i}\left[-\frac{1}{4\lambda_{1i}}+\frac{1}{V_i^*}+\frac{\textcolor{black}{p}}{2\sqrt{gH_i^*}}\right],\\
				&\gamma_{2i}(x)=-\frac{(H_i^*)_x}{H_i^*}\lambda_{1i}\lambda_{2i}\left[\frac{1}{4\lambda_{2i}}+\frac{1}{V_i^*}-\frac{\textcolor{black}{p}}{2\sqrt{gH_i^*}}\right],\\
				&\delta_{2i}(x)=-\frac{(H_i^*)_x}{H_i^*}\lambda_{1i}\lambda_{2i}\left[\frac{3}{4\lambda_{2i}}+\frac{1}{V_i^*}+\frac{\textcolor{black}{p}}{2\sqrt{gH_i^*}}\right],
			\end{aligned}
		\end{equation}
Using the definition of $\varphi_i(x)$ in \eqref{varphii}, we can simplify the integrand in the left-hand side of \eqref{intergaline} as
\begin{equation}\label{tem}
   e^{\int_{0}^{s}2\frac{\gamma_{2i}}{\lambda_{1i}}d\xi}\frac{\gamma_{2i}(s)}{\lambda_{2i}(s) \varphi_i(s)} = e^{\int_{0}^{s}\left(\frac{2\gamma_{2i}-\gamma_{1i}}{\lambda_{1i}}-\frac{\delta_{2i}}{\lambda_{2i}}\right)d\xi}\frac{\gamma_{2i}(s)}{\lambda_{2i}(s)}.
\end{equation}
Here and hereafter, we will omit the integration variables for the sake of simplicity. 
From \eqref{mall>0}, the exponential function part in \eqref{tem} can be rewritten as
\begin{equation}\label{ovc}
\begin{aligned}
    e^{\int_{0}^{s}\left(\frac{2\gamma_{2i}-\gamma_{1i}}{\lambda_{1i}}-\frac{\delta_{2i}}{\lambda_{2i}}\right)d\xi}
    &= e^{\int_{0}^{s}\frac{1}{\lambda_{1i}\lambda_{2i}}\left(\lambda_{2i}(2\gamma_{2i}-\gamma_{1i})-\lambda_{1i}\delta_{2i} \right) d\xi}   \\
    &= e^{\int_{0}^{s}\frac{(H_i^*)_{\xi}}{\lambda_{1i}\lambda_{2i}}\left(\frac{\textcolor{black}{3+2p}}{2}\left(\frac{gH_i^*-V_i^{*2}}{H_i^*}\right)+\frac{3\sqrt{gH_i^*}V_i^*}{H_i^*}\right)d\xi}\\
    &=e^{\int_{0}^{s} \frac{\textcolor{black}{3+2p}}{2}\frac{(H_i^*)_{\xi}}{H_i^*}d\xi}\cdot e^{\int_{0}^{s}\frac{3\sqrt{gH_i^*}V_i^*}{H_i^*(gH_i^{*}-V_i^{*2})}(H_i^*)_{\xi}d\xi}\\
    &=\left( \frac{H_i^*(s)}{H_i^*(0)}\right)^{\frac{\textcolor{black}{3+2p}}{2}} \exp\left(\displaystyle\int_{0}^{s}\left( \frac{(H_i^*)_{\xi}}{H_i^*}+
    \underbrace{\frac{\left( \frac{\sqrt{gH_i^*}}{2H_i^*}+\frac{V_i^*}{H_i^*}\right)}{\sqrt{gH_i^*}-V_i^*}(H_i^*)_{\xi}}_{I_1}+
    \underbrace{\frac{-\frac{3\sqrt{gH_i^*}}{2H^*_{i}}}{\sqrt{gH_i^*}+V_i^*}(H_i^*)_{\xi}}_{I_2}\right)d\xi\right).
    \end{aligned}
    \end{equation}
From \eqref{lam1} and using \eqref{nve}, we can obtain 
\begin{equation}\label{I1}
I_1=\frac{(\lambda_{2i})_{\xi}}{\lambda_{2i}}.
\end{equation}
On the other hand,
\begin{equation}\label{I2}
I_2=\frac{-\frac{3}{2}\sqrt{gH_i^*}(H_i^*)_{\xi}}{{\sqrt{g}(H^*_i)^{\frac{3}{2}}+Q_i}}=-\frac{({\sqrt{g}(H^*_i)^{\frac{3}{2}}+Q_i})_{\xi}}{{\sqrt{g}(H^*_i)^{\frac{3}{2}}+Q_i}}.
\end{equation}
Combining \eqref{ovc}--\eqref{I2}, one has
    \begin{equation}\label{comb1}
    \begin{aligned}
    e^{\int_{0}^{s}\left(\frac{2\gamma_{2i}-\gamma_{1i}}{\lambda_{1i}}-\frac{\delta_{2i}}{\lambda_{2i}}\right)d\xi}
    &=\left( \frac{H_i^*(s)}{H_i^*(0)}\right)^{\frac{\textcolor{black}{5+2p}}{2}}\cdot \frac{\lambda_{2i}(s)}{\lambda_{2i}(0)} \cdot \frac{\sqrt{g}(H_i^*(0))^{\frac{3}{2}}+Q_i}{\sqrt{g}(H_i^*(s))^{\frac{3}{2}}+Q_i}\\
    &=\left( \frac{H_i^*(s)}{H_i^*(0)}\right)^{\frac{\textcolor{black}{5+2p}}{2}}\cdot \frac{\lambda_{2i}(s)}{\lambda_{2i}(0)} \cdot \frac{\sqrt{g}(H_i^*(0))^{\frac{3}{2}}+H_i^*(0)V_i^*(0)}{\sqrt{g}(H_i^*(s))^{\frac{3}{2}}+H_i^*(s)V_i^*(s)}\\
    &=\left(\frac{H_i^*(s)}{H_i^*(0)}\right)^{\frac{\textcolor{black}{5+2p}}{2}}\cdot \frac{\lambda_{2i}(s)}{\lambda_{2i}(0)} \cdot \frac{H_i^*(0)\lambda_{1i}(0)}{H_i^*(s)\lambda_{1i}(s)} \\
    &=\left( \frac{H_i^*(s)}{H_i^*(0)}\right)^{\frac{\textcolor{black}{3+2p}}{2}}\cdot \frac{\lambda_{2i}(s)}{\lambda_{2i}(0)} \cdot \frac{\lambda_{1i}(0)}{\lambda_{1i}(s)},
    \end{aligned}
    \end{equation}
which together with 
\begin{equation}\label{comb2}
 \frac{\gamma_{2i}}{\lambda_{2i}}
   =-\frac{\textcolor{black}{3-2p}}{4H_i^*}(H_i^*)_x-\frac{\sqrt{gH_i^*}}{Q_i}(H_i^*)_x+\frac{\textcolor{black}{p}Q_i}{2\sqrt{gH_i^*}H_i^{*2}}(H_i^*)_x-\frac{\sqrt{gH_i^*}}{2H_i^*(\sqrt{gH_i^*}-V_i^*)}(H_i^*)_x.
\end{equation}
gives
\begin{equation}\label{exb}
\begin{aligned}
    f(s):&=e^{\int_{0}^{s}\left(\frac{2\gamma_{2i}-\gamma_{1i}}{\lambda_{1i}}-\frac{\delta_{2i}}{\lambda_{2i}}\right)d\xi}\frac{\gamma_{2i}}{\lambda_{2i}} \\
     &=\left( \frac{H_i^*(s)}{H_i^*(0)}\right)^{\frac{\textcolor{black}{3+2p}}{2}}\cdot \frac{\lambda_{2i}(s)}{\lambda_{2i}(0)} \cdot \frac{\lambda_{1i}(0)}{\lambda_{1i}(s)}\cdot (H_i^*)_s\left[-\frac{\textcolor{black}{3-2p}}{4H_i^*}-\frac{\sqrt{gH_i^*}}{Q_i}+\frac{\textcolor{black}{p}Q_i}{2\sqrt{gH_i^*}H_i^{*2}}-\frac{\sqrt{gH_i^*}}{2H_i^*(\sqrt{gH_i^*}-V_i^*)} \right]\\
     &<\left( \frac{H_i^*(s)}{H_i^*(0)}\right)^{\frac{\textcolor{black}{3+2p}}{2}}\cdot \frac{\lambda_{2i}(s)}{\lambda_{2i}(0)} \cdot \frac{\lambda_{1i}(0)}{\lambda_{1i}(s)}\cdot (H_i^*)_s\left[-\frac{\textcolor{black}{3-2p}}{4H_i^*}-\frac{\sqrt{gH_i^*}}{Q_i}-\frac{\sqrt{gH_i^*}}{2H_i^*(\sqrt{gH_i^*}-V_i^*)} \right].
    \end{aligned}
\end{equation}
Next, we estimate the terms on the right-hand side of \eqref{exb} one by one.
Firstly, we focus on $\frac{\lambda_{2i}(s)}{\lambda_{2i}(0)} \cdot \frac{\lambda_{1i}(0)}{\lambda_{1i}(s)}$, since
\begin{equation}
\begin{aligned}
    &\lambda_{2i}(s)\lambda_{1i}(0)-\lambda_{2i}(0)\lambda_{1i}(s)\\
    =&\left(\sqrt{gH_i^*(s)}-V_i^*(s)\right)\left(\sqrt{gH_i^*(0)}+V_i^*(0)\right)-\left(\sqrt{gH_i^*(0)}-V_i^*(0)\right)\left(\sqrt{gH_i^*(s)}+V_i^*(s)\right) \\
=&2\sqrt{g}\left(\sqrt{H_i^*(s)}V_i^*(0)-\sqrt{H_i^*(0)}V_i^*(s)\right)\\
    =&2\sqrt{g}Q_{i}\left(\frac{\sqrt{H_i^*(s)}H_i^*(s)-\sqrt{H_i^*(0)}H_i^*(0)}{H_i^*(0)H_i^*(s)} \right).
    \end{aligned}
\end{equation}
From \eqref{steady-state}, $H_i^*$ is a monotonically decreasing function,  it follows immediately that
\begin{equation}\label{ine02}
    \frac{\lambda_{2i}(s)}{\lambda_{2i}(0)} \cdot \frac{\lambda_{1i}(0)}{\lambda_{1i}(s)}<1.
\end{equation}
Next, we focus on $\frac{H_i^*(s)}{H_i^*(0)}\cdot \frac{\lambda_{1i}(0)}{\lambda_{1i}(s)}$ and obtain that $\frac{\lambda_{1i}}{H_{i}^{*}}= \sqrt{\frac{g}{H_{i}^{*}}}+\frac{Q_i}{(H^{*}_{i})^{2}}$ is an increasing function since $H^{*}_{i}$ is decreasing. As a consequence
\begin{equation}\label{ine01}
  \frac{H_i^*(s)}{H_i^*(0)}\cdot \frac{\lambda_{1i}(0)}{\lambda_{1i}(s)}<1.
\end{equation}

By using the relationship \eqref{ine01}, and recalling that $(H_{i}^{*})_{x}<0$, we know that
\begin{equation}
    \begin{aligned}
        &-
        \left( \frac{H_i^*(s)}{H_i^*(0)}\right)^{\frac{\textcolor{black}{3+2p}}{2}}\cdot \frac{\lambda_{2i}(s)}{\lambda_{2i}(0)} \cdot \frac{\lambda_{1i}(0)}{\lambda_{1i}(s)}\frac{\sqrt{gH_i^*(s)}}{2H_i^*(s)(\sqrt{gH_i^*(s)}-V_i^*(s))}(H_{i}^{*})_{s}
        \\
        <& -
        \left( \frac{H_i^*(s)}{H_i^*(0)}\right)^{\frac{\textcolor{black}{1+2p}}{2}}\cdot \frac{\lambda_{2i}(s)}{\lambda_{2i}(0)} \frac{\sqrt{gH_i^*(s)}}{2H_i^*(s)(\sqrt{gH_i^*(s)}-V_i^*(s))}(H_{i}^{*})_{s}
        \\
        =&-
        \left( \frac{H_i^*(s)}{H_i^*(0)}\right)^{\frac{\textcolor{black}{1+2p}}{2}}\cdot \frac{1}{\lambda_{2i}(0)} \frac{\sqrt{gH_i^*(s)}}{2H_i^*(s)}(H_{i}^{*})_{s}.
    \end{aligned}
\end{equation}

Again from \eqref{steady-state},  we know that $H_i^*(x)$ is a decreasing function, therefore $\frac{H_i^*(s)}{H_i^*(0)}<1$, this together with \eqref{ine02} gives
\begin{equation}
    \begin{aligned}
        f(s)<-\left\{\left( \frac{H_i^*(s)}{H_i^*(0)}\right)^{\frac{\textcolor{black}{1+2p}}{2}}\frac{\textcolor{black}{\max\{0,3-2p\}}}{4H_i^*(0)}+\left( \frac{H_i^*(s)}{H_i^*(0)}\right)^{\frac{\textcolor{black}{3+2p}}{2}} \frac{\sqrt{gH_i^*(s)}}{Q_i}+\left( \frac{H_i^*(s)}{H_i^*(0)}\right)^{\frac{\textcolor{black}{1+2p}}{2}} \frac{1}{\lambda_{2i}(0)} \frac{\sqrt{gH_i^*(s)}}{2H_i^*(s)}\right\}(H_i^*)_s.
    \end{aligned}
\end{equation}
As a consequence,
\begin{equation}\label{f(s)}
    \begin{aligned}
       \int_{0}^{x^i_{0}}f(s) ds<&
        -\int_{H^*_{i}(0)}^{H^*_{i}(x^i_{0})}
        \left( \frac{H_i^*}{H_i^*(0)}\right)^{\frac{\textcolor{black}{1+2p}}{2}}\frac{\textcolor{black}{\max\{0,3-2p\}}}{4H_i^*(0)}dH_i^*-
        \int_{H^*_{i}(0)}^{H^*_{i}(x^i_{0})}
        \left( \frac{H_i^*}{H_i^*(0)}\right)^{\frac{\textcolor{black}{3+2p}}{2}}\frac{\sqrt{gH_i^*}}{Q_i}dH_i^*\\
        &-
        \int_{H^*_{i}(0)}^{H^*_{i}(x^i_{0})}
        \left( \frac{H_i^*}{H_i^*(0)}\right)^{\frac{\textcolor{black}{1+2p}}{2}}\frac{1}{\lambda_{2i}(0)} \frac{\sqrt{gH_i^*}}{2H_i^*}dH_i^*\\
        =&\textcolor{black}{\frac{\textcolor{black}{\max\{0,3-2p\}}}{2(3+2p)}}\left(\frac{1}{H_i^*(0)} \right)^{\frac{\textcolor{black}{3+2p}}{2}}\left((H_i^*(0))^{\frac{\textcolor{black}{3+2p}}{2}}-(H_i^*(x^i_0))^{\frac{\textcolor{black}{3+2p}}{2}} \right)\\
        &+\frac{\sqrt{g}}{\textcolor{black}{(3+p)}Q_i}\left(\frac{1}{H_i^*(0)} \right)^{\frac{\textcolor{black}{3+2p}}{2}}\left( (H_i^*(0))^{\textcolor{black}{3+p}}-(H_i^*(x^i_0))^{\textcolor{black}{3+p}} \right)\\
        &+\frac{\sqrt{g}}{\textcolor{black}{2(1+p)}\lambda_{2i}(0)}\left(\frac{1}{H_i^*(0)}\right)^{\frac{\textcolor{black}{1+2p}}{2}}\left( (H_i^{*}(0))^{\textcolor{black}{1+p}}-(H_i^{*}(x^i_0))^{\textcolor{black}{1+p}}\right).
    \end{aligned}
\end{equation}
We denote $d_{i} = H_{i}^*(x^i_{0})/H_{i}^{*}(0)$, then $d_i\in(0,1)$, $i\in \{1,2,\cdots,n\}$.
We now express the three terms in \eqref{f(s)} using $d_i$ as the following
\begin{equation}
\textcolor{black}{\frac{\textcolor{black}{\max\{0,3-2p\}}}{2(3+2p)}}\left(\frac{1}{H_i^*(0)} \right)^{\frac{\textcolor{black}{3+2p}}{2}}\left((H_i^*(0))^{\frac{\textcolor{black}{3+2p}}{2}}-(H_i^*(x^{i}_0))^{\frac{\textcolor{black}{3+2p}}{2}} \right)= \textcolor{black}{\frac{\textcolor{black}{\max\{0,3-2p\}}}{2(3+2p)}}\left(1-d_i^{\frac{\textcolor{black}{3+2p}}{2}}\right).
\end{equation}
 Since the critical points satisfy \eqref{dcr}, we have
\begin{equation}
\begin{split}
   & \frac{\sqrt{g}}{\textcolor{black}{(3+p)}Q_i}\left(\frac{1}{H_i^*(0)} \right)^{\frac{\textcolor{black}{3+2p}}{2}}\left( (H_i^*(0))^{\textcolor{black}{3+p}}-(H_i^*(x^{i}_0))^{\textcolor{black}{3+p}} \right)\\
    =&\frac{1}{\textcolor{black}{3+p}}\left(\frac{1}{H_i^*(x^i_0)}\right)^{\frac{3}{2}}\left(\frac{1}{H_i^*(0)} \right)^{\frac{\textcolor{black}{3+2p}}{2}}\left( (H_i^*(0))^{\textcolor{black}{3+p}}-(H_i^*(x^{i}_0))^{\textcolor{black}{3+p}} \right)\\
=&\frac{1}{\textcolor{black}{3+p}}\left( d_i^{-\frac{3}{2}}-d_i^{\frac{\textcolor{black}{3+2p}}{2}}\right).
\end{split}
\end{equation}
And using the definition of the critical points \eqref{dcr} again, we obtain
\begin{equation}
\begin{split}
&\frac{\sqrt{g}}{\textcolor{black}{2(1+p)}\lambda_{2i}(0)}\left(\frac{1}{H_i^*(0)}\right)^{\frac{\textcolor{black}{1+2p}}{2}}\left( (H_i^{*}(0))^{\textcolor{black}{1+p}}-(H_i^{*}(x^i_0))^{\textcolor{black}{1+p}}\right)\\
=&\frac{\sqrt{g}}{\textcolor{black}{2(1+p)}(\sqrt{gH_i^*(0)}-Q_i(H_i^*(0))^{-1})(H_i^*(0))^{\frac{\textcolor{black}{1+2p}}{2}}}\left( (H_i^{*}(0))^{\textcolor{black}{1+p}}-(H_i^{*}(x_0))^{\textcolor{black}{1+p}}\right) \\
=&\frac{1}{\textcolor{black}{2(1+p)}((H_i^*(0))^{\textcolor{black}{1+p}}-(H_i^*(x^i_0))^{\frac{3}{2}}(H_i^*(0))^{\textcolor{black}{\frac{2p-1}{2}}})}\left( (H_i^{*}(0))^{\textcolor{black}{1+p}}-(H_i^{*}(x_0))^{\textcolor{black}{1+p}}\right)\\ 
=& \frac{1}{\textcolor{black}{2(1+p)}}\left(\frac{1-d_i^{\textcolor{black}{1+p}}}{1-d_i^{\frac{3}{2}}} \right).
\end{split}
\end{equation}

One the other hand, we have
\begin{equation}
\begin{split}
    \frac{\lambda_{1}(0)}{\lambda_{1}(0)-\lambda_{2}(0)}=&\frac{\sqrt{g (H^*_{i}(0))^{3}}+Q_i}{2Q_i}\\
    =& \frac{1}{2}\left(1+(H^*_{i}(0))^{\frac{3}{2}}\frac{\sqrt{g}}{Q_i}\right) \\
    =&\frac{1}{2}\left(1+(H^*_{i}(0))^{\frac{3}{2}}(H_i^*(x^0_i))^{-\frac{3}{2}})\right)\\
    =&\frac{1}{2}\left(1+d_i^{-\frac{3}{2}}\right).
\end{split}
\end{equation}
So overall, to show \eqref{intergaline} it is sufficient to show that
\begin{equation}
  \textcolor{black}{\frac{\textcolor{black}{\max\{0,3-2p\}}}{2(3+2p)}}\left(1-d_i^{\frac{\textcolor{black}{3+2p}}{2}}\right)+\frac{1}{\textcolor{black}{3+p}}\left( d_i^{-\frac{3}{2}}-d_i^{\frac{\textcolor{black}{3+2p}}{2}}\right)+ \frac{1}{\textcolor{black}{2(1+p)}}\left(\frac{1-d_i^{\textcolor{black}{1+p}}}{1-d_i^{\frac{3}{2}}} \right)<\frac{1}{2}\left(1+d_i^{-\frac{3}{2}}\right),
\end{equation}
which now only depends on $d_i$ and the parameter $p$.

Multiplying $1-d_i^{\frac{3}{2}}$ to the both sides of the above inequality, we obtain
\begin{equation}
\label{eq:twocases}
    \begin{aligned}
        \frac{\textcolor{black}{\max\{0,\!3\!-\!2p\}}}{\textcolor{black}{3\!+\!2p}}\!\left(\! 1\!-d_i^{\frac{3}{2}}\!\!-d_i^{\frac{\textcolor{black}{3+2p}}{2}}\!\!\!+d_i^{\textcolor{black}{3+p}}\!\right)\!\!+\!\frac{2}{\textcolor{black}{3+p}}\!\left(\!\! -1\!+\!d_i^{-\frac{3}{2}}\!\!-d_i^{\frac{\textcolor{black}{3+2p}}{2}}\!\!\!+d_i^{\textcolor{black}{3+p}}\right)\!+\!\frac{1}{\textcolor{black}{1+p}}\!\left(\!1\!-\!d_i^{\textcolor{black}{1+p}} \!\right)\!<\!\!-d_i^{\frac{3}{2}}\!+\!d_i^{-\frac{3}{2}}.
    \end{aligned}
\end{equation}
Assume for now that $p\leq 3/2$, then the above
{\color{black}
is equivalent to
\begin{equation}
\label{eq:equiv1}
    (1\!-\!d_i^{\frac{3+2p}{2}}\!\!+d_i^{3+p})\!-\!\frac{4p}{3+2p}(1\!-d_i^{\frac{3}{2}}\!-d_i^{\frac{3+2p}{2}}\!\!+d_i^{3+p})+\frac{2}{3+p}\left(\!-1\!-d_i^{\frac{3+2p}{2}}\!+d_i^{3+p} \right)\!+\!\frac{1}{1+p}(1\!-d_i^{1+p})<\frac{1+p}{3+p}d_i^{-\frac{3}{2}}.
\end{equation}
Denote 
\begin{equation}\label{dFR}
\begin{split}
F(x,p) =& (1-x^{\frac{3+2p}{2}}+x^{3+p})-\frac{4p}{3+2p}(1-x^{\frac{3}{2}}-x^{\frac{3+2p}{2}}+x^{3+p})-\frac{2}{3+p}\\
&+\left( \frac{2}{3+p}-1\right)x^{-\frac{3}{2}}+\frac{2}{3+p}\left(-x^{\frac{3+2p}{2}}+x^{3+p} \right)+\frac{1}{1+p}(1-x^{1+p}).
\end{split}
\end{equation}
The equation \eqref{eq:equiv1} is equivalent to show that $F(d_{i},p)<0$. We rewrite \eqref{dFR} as
\begin{equation*}
F(x,p)=
\Big(\frac{2}{3+p}-1\Big)\big(x^{-\frac{3}{2}}-1\big)
+\Big(\frac{4p}{3+2p}\Big)\big(x^{\frac{3}{2}}-1\big)
+\Big(1-\frac{4p}{3+2p}+\frac{2}{3+p}\Big)\big(x^{3+p}-x^{\frac{3+2p}{2}}\big)
+\frac{1}{1+p}\big(1-x^{1+p}\big).
\end{equation*}
Now observe that for every $p\geq 0$ and every $x\in(0,1)$, using the mean value theorem with some $c\in(x,1)$
\begin{equation}
\label{eq:meanval}
\frac{1}{1+p}\big(1-x^{1+p}\big) = \int_{x}^{1}s^{p} ds = c^{p}(1-x)\ \le\ 1-x\ \le\ 1-x^{\frac{3}{2}}.
\end{equation}
Let us now define
\[
B(p):=\frac{4p}{3+2p},\qquad
C(p):=1-\frac{4p}{3+2p}+\frac{2}{3+p},\qquad
K(p):=-\left( \frac{2}{3+p}-1\right) = \frac{1+p}{3+p}.
\]
One has, using \eqref{eq:meanval},
\begin{equation}
\label{eq:F1}
\begin{split}
F(x,p) \leq& -K(p)\big(x^{-\frac{3}{2}}-1\big)
+B(p)\big(x^{\frac{3}{2}}-1\big)
+C(p)\big(x^{3+p}-x^{\frac{3+2p}{2}}\big)
+\big(1-x^{\frac{3}{2}}\big)\\
=&(1-x^{\frac{3}{2}})\left[1-\left(K(p)x^{-\frac{3}{2}}
+B(p)
+C(p)x^{\frac{3+2p}{2}}\right)\right].
\end{split}
\end{equation}
Let us set $H(p,x) := K(p)x^{-3}+B(p)+C(p)x^{3+2p}$, we are going to show that $H(p,d_{i}^{1/2})>1$. First let us note that for $p\geq 0$ given, two possible cases can occur: either $C(p)> 0$ or $C(p)\leq 0$. If $C(p)\leq 0$ then $H(p,\cdot)$ admits a minimum on $(0,1]$ at $x=1$ and thus
\begin{equation}
H(p,x) \geq H(p,1) = K(p)+B(p)+C(p) = 2 > 1,\;\forall x\in(0,1).
\end{equation}
If $C(p)>0$ then, one only needs to check the cases where $p\textcolor{black}{<} 3/4$. Indeed, if $p\geq 3/4$ then $B(p) \geq 2/3$ and since $K(p)x^{-3}>(1/3)x^{-3}>1/3$ and  $C(p)>0$ then $H(p,x)>1$ for any $x\in(0,1)$. Let us now consider the case $p\in[0,3/4)$. In this case, using the definitions of $B(p)$ and $C(p)$
\begin{equation}
H(p,x) \textcolor{black}{=} B(p)(1-x^{3+2p}) + \left(1+\frac{2}{3+p}\right)x^{3+2p} + K(p)x^{-3}.
\end{equation}
Noting that $K(p)>1/3$, and that $x\in(0,1)$, we have
\begin{equation}
H(p,x) > \left(1+\frac{8}{15}\right)x^{\frac{9}{2}} + \frac{x^{-3}}{3}.
\end{equation}
This expression admits a minimum on $x\in(0,1)$ which satisfies $x^{15/2} = 2/(9\alpha)$ with $\alpha = 23/15$, thus
\begin{equation}
H(p,x) >\alpha\left(\frac{2}{9\alpha}\right)^{\frac{9}{15}} + \frac{1}{3}\left(\frac{2}{9\alpha}\right)^{-\frac{6}{15}} = \left[\left(\frac{2}{9}\right)^{\frac{3}{5}} + \frac{1}{3}\left(\frac{9}{2}\right)^{\textcolor{black}{\frac{2}{5}}}\right]\left(\frac{23}{15}\right)^{\frac{2}{5}} >1.
\end{equation}
Thus for any $p\geq 0$ and any $x\in(0,1)$, $H(p,x)>1$ and therefore, from \eqref{eq:F1}, $F(d_{i},p)<0$ and therefore \eqref{eq:equiv1} holds.

Let us now come back to \eqref{eq:twocases} and deal with the last case: $p>3/2$. In this case \eqref{eq:twocases} becomes
\begin{equation}
\label{eq:twocases2}
    \begin{aligned}
\frac{2}{\textcolor{black}{3+p}}\left( -1+d_i^{-\frac{3}{2}}-d_i^{\frac{\textcolor{black}{3+2p}}{2}}+d_i^{\textcolor{black}{3+p}}\right)+\frac{1}{\textcolor{black}{1+p}}\left(1-d_i^{\textcolor{black}{1+p}} \right)<-d_i^{\frac{3}{2}}+d_i^{-\frac{3}{2}}.
    \end{aligned}
\end{equation}
Using that $d^{3+p} = d^{3/2+p}d^{3/2}$, this is equivalent to
\begin{equation}
\label{eq:twocases3}
    \begin{aligned}
    \frac{p+1}{p+3}(d_{i}^{-\frac{3}{2}}-1)+(1-d_{i}^{\frac{3}{2}})\left(1+\frac{2}{p+3}d_{i}^{\frac{3+2p}{2}}\right)-\frac{1}{1+p}\left(1-d_{i}^{p+1}\right)>0.
    \end{aligned}
\end{equation}
Then observe that the function $y\rightarrow (1-d^{y})y^{-1}$ is decreasing for $y>1$ (this follows from the fact that $e^{x}\geq 1+x$ and therefore $1\geq [1-\ln(d)y]d^{y}$ for $d\in(0,1)$). Therefore, applying this with $p+1\geq 3/2$ we obtain
\begin{equation}
    \frac{1}{1+p}(1-d^{p+1})\leq \frac{2}{3}(1-d^{\frac{3}{2}})
\end{equation}
and to show \eqref{eq:twocases3} (hence \eqref{eq:twocases} for $p>3/2$), it thus suffices to show that
\begin{equation}
    \begin{aligned}
    \frac{p+1}{p+3}(d_{i}^{-\frac{3}{2}}-1)+(1-d_{i}^{\frac{3}{2}})\left(\frac{1}{3}+\frac{2}{p+3}d_{i}^{\frac{3+2p}{2}}\right)>0,
    \end{aligned}
\end{equation}
which holds since all the terms are positive.
}
The proof of the inequality \eqref{intergaline} is complete.
\end{proof}
Based on Lemmas \ref{etaresult} and \ref{lem1}, we can prove the following
 \begin{lem}\label{lem2}
			For $i\in\{1,2,\cdots,n\}$, let $x^{i}_0$ be the critical point defined by \eqref{x0}, 
			the solution $\bar{\eta}_i(x)$ to
			\begin{equation}\label{bareta}
				\left\{
				\begin{aligned}
					&\bar{\eta}'_i=\left| \frac{\delta_{1i}\varphi_i}{\lambda_{1i}}+\frac{\gamma_{2i}}{\lambda_{2i}\varphi_i}\bar{\eta}^2_i \right|, \\
					&\bar{\eta}_i(0)=1,
				\end{aligned}
				\right.
			\end{equation}
exists on $[0,x^{i}_0)$.
		\end{lem}
		
		\begin{proof}			
    It can be shown from \eqref{varphii} that $\lambda_{1i}>\lambda_{2i}>0$. From \eqref{def02}, $\varphi_i>0$ and thanks to \eqref{positive}, we have 
    \begin{equation}
			\label{eq:etapositive}
   \frac{\delta_{1i}\varphi_i}{\lambda_{1i}}+\frac{\gamma_{2i}}{\lambda_{2i}\varphi_i}\bar{\eta}^2_i >0.
   \end{equation}
   Denote by \begin{equation}\label{GiI}
   G_i(x):=\frac{\delta_{1i}\varphi_i}{\lambda_{1i}},\quad I_i(x):=\frac{\gamma_{2i}}{\lambda_{2i}\varphi_i},
   \end{equation} 
   we rewrite the equation in \eqref{bareta} as
            \begin{equation}\label{ebareta1}
					\bar{\eta}'_i=G_i(x)+I_i(x)\bar{\eta}^2_i.
			\end{equation}
            Note that \eqref{ebareta1} is a Riccati equation. Since, from Lemma \ref{etaresult} that $\eta_{0i}$ is a special solution to \eqref{ebareta1}, 
            thus the general solution to \eqref{ebareta1} can be expressed as
\begin{equation}\label{change1}
\bar\eta_i=u_i+\eta_{0i}.
\end{equation}
Substituting \eqref{change1} to equation \eqref{ebareta1}, one has that $u_i$ should satisfy
\begin{equation}\label{Bernoulli}
u_i'(x)=I_i(x)u_i^2(x)+2I_i(x)\eta_{0i}(x)u_i(x)
\end{equation}
which is a Bernoulli equation. Moreover, from \eqref{bareta}, 
\begin{equation}
u_{i}(0) = \bar{\eta}_{i}'(0)-\frac{\lambda_{2i}(0)}{\lambda_{1i}(0)}>0.
\end{equation}
By comparison principle (see \cite{Hartman}), one has $u_{i}(x)>0$ as long as $u_i(x)$ exists (note that $\eta_{0i}$ is a positive function).
We now solve \eqref{Bernoulli} as follows. Let 
\begin{equation}\label{change2}
z_i=\displaystyle\frac{1}{u_i},
\end{equation}
then $u_{i}$ is a solution to \eqref{Bernoulli} if and only if $z_i$ satisfies
\begin{equation}\label{equ-z}
z_i'(x)+2I_i(x)\eta_{0i}z_i(x)=-I_i(x).
\end{equation}
The solution to \eqref{equ-z} is
\begin{equation}\label{sol-z}
z_i(x)=e^{-\int_{0}^{x}2I_i(\xi)\eta_{0i}(\xi)d\xi}\left(z_i(0)-\int_{0}^{x}e^{\int_{0}^{s}2I_i(\xi)\eta_{0i}(\xi)d\xi}I_i(s)ds\right).
\end{equation}
Noticing \eqref{change1} and \eqref{change2}, this means that as long as 
  \begin{equation}
       -\int_{0}^{x}e^{\int_{0}^{s}2I_i(\xi)\eta_{0i}(\xi)d\xi}I_i(s)ds+\frac{1}{\bar\eta_i(0)-\eta_{0i}(0)}\neq 0,
   \end{equation}
   the solution to \eqref{ebareta1} is 
\begin{equation}\label{baretafunction}
\bar{\eta}_i(x)=\eta_{0i}(x)+e^{\int_{0}^{x}2I_i(\xi)\eta_{0i}(\xi)d\xi}\frac{1}{-\int_{0}^{x}e^{\int_{0}^{s}2I_i(\xi)\eta_{0i}(\xi)d\xi}I_i(s)ds+\frac{1}{\bar\eta_i(0)-\eta_{0i}(0)}}.
\end{equation}
   We observe that $-\int_{0}^{x}e^{\int_{0}^{s}2I_i(\xi)\eta_{0i}(\xi)d\xi}I_i(s)ds$ is a monotonically decreasing function of $x$ with zero as its maximum value (note that $\gamma_{2i} >0$ from \eqref{positive}, thus $I_{i}>0$ from \eqref{GiI}).
   Noticing the expression of $\eta_{0i}$ in \eqref{eta0i}, one has
   \begin{equation}
   \frac{1}{\bar\eta_i(0)-\eta_{0i}(0)}=\frac{\lambda_{1i}(0)}{\lambda_{1i}(0)-\lambda_{2i}(0)}.
   \end{equation}
Moreover, noticing the expressions of $\eta_{0i}$ in \eqref{eta0i} and $I_i$ in \eqref{GiI}, we obtain that the solution to \eqref{bareta}, i.e.
\begin{equation}\label{exbareta}
\bar{\eta}_i(x)=\frac{\lambda_{2i}(x)}{\lambda_{1i}(x)}\varphi_i(x)+\frac{e^{\int_{0}^{x}2\frac{\gamma_{2i}(\xi)}{\lambda_{1i}(\xi)}d\xi}}{-\int_{0}^{x}e^{\int_{0}^{s}2\frac{\gamma_{2i}(\xi)}{\lambda_{1i}(\xi)}d\xi}\frac{\gamma_{2i}(s)}{\lambda_{2i}(s)\varphi_i(s)}ds+\frac{\lambda_{1i}(0)}{\lambda_{1i}(0)-\lambda_{2i}(0)}}
\end{equation}
exists on $x \in [0, x^{i}_0)$ thanks to Lemma \ref{lem1}.
                \end{proof}
\begin{rmk}\label{remf}
    Using the equivalent expression \eqref{mall>0} of $\gamma_{1i}$, $\gamma_{2i}$, $\delta_{1i}$ and $\delta_{2i}$, and the fact that $H_{i}^{*}V_{i}^{*} = Q_{i}$, a key nontrivial observation is that 
    the expression of \eqref{exbareta} can be further simplified as 
    \begin{equation}\label{explicitetai}
       \bar\eta_i(x)=m_i(x)\frac{\lambda_{2i}(x)}{\lambda_{1i}(x)}\varphi_i(x)=m_i(x)\eta_{0i}(x),
    \end{equation}
  where 
    \begin{equation}\label{mx}
  m_i(x)=\frac{\frac{1}{2}(H_i^{*}(x))^{\frac{\textcolor{black}{3+2p}}{2}}+\frac{\sqrt{g}}{\textcolor{black}{(3+p)}Q_i}(H_i^{*}(x))^{\textcolor{black}{3+p}}+\frac{\textcolor{black}{(1+p)}\sqrt{g}}{\textcolor{black}{2(3+p)}Q_i}(H_i^{*}(0))^{\textcolor{black}{3+p}}+\frac{Q_i}{2\sqrt{g}}((H^*_i(x))^{\textcolor{black}{p}}-(H^*_i(0))^{\textcolor{black}{p}})}{-\frac{1}{2}(H_i^{*}(x))^{\frac{\textcolor{black}{3+2p}}{2}}+\frac{\sqrt{g}}{\textcolor{black}{(3+p)}Q_i}(H_i^{*}(x))^{\textcolor{black}{3+p}}+\frac{\textcolor{black}{(1+p)}\sqrt{g}}{\textcolor{black}{2(3+p)}Q_i}(H_i^{*}(0))^{\textcolor{black}{3+p}}+\frac{Q_i}{2\sqrt{g}}((H^*_i(x))^{\textcolor{black}{p}}-(H^*_i(0))^{\textcolor{black}{p}})}
  \end{equation}
  We leave the proof to Appendix \ref{app:nontrivial}.
\end{rmk}
               With the existence of $\bar{\eta}_i$ on $[0,x_0^i)$, we can show the following
\begin{lem}
\label{lem:cj}
For $i\in\{1,2,\cdots,n\}$, one has
      \begin{equation}
      \varphi_i(x)>\bar{\eta}_i(x),\quad \forall x\in(0,x^i_0),
\end{equation}
where $\bar\eta_i$ is the solution to \eqref{bareta}
  and $\varphi_i$ is defined in \eqref{varphii}.
\end{lem}

\begin{proof}
The proof of this lemma is based on a comparison argument. We will drop the index $i$ for clarity given that $i$ is fixed all along the proof. Note first that $\varphi(0) = \bar{\eta}(0)=1$. Note also that, from \eqref{varphii},
\begin{equation}
\label{eq:varphi'}
    \varphi'= \left(\frac{\gamma_{1}}{\lambda_{1}}+\frac{\delta_{2}}{\lambda_{2}}\right)\varphi,
\end{equation}
and from \eqref{bareta} together with \eqref{eq:etapositive},
\begin{equation}
\label{eq:eta'}
    \bar{\eta}' = \left(\frac{\delta_{1}}{\lambda_{1}}+\frac{\gamma_{2}}{\lambda_{2}}\frac{\bar{\eta}^{2}}{\varphi^{2}}\right)\varphi.
\end{equation}
Let us compare the two right-hand sides of \eqref{eq:varphi'} and \eqref{eq:eta'}: from \eqref{all>0}, we have
\begin{equation}
\label{eq:comparison}
\begin{split}
\frac{\gamma_{1}}{\lambda_{1}}+\frac{\delta_{2}}{\lambda_{2}} - \left(\frac{\delta_{1}}{\lambda_{1}}+\frac{\gamma_{2}}{\lambda_{2}}\right) 
=&\frac{gCV^{*2}}{H^*} \left[\frac{1}{2}\left(\frac{1}{\lambda_{2}^{2}}-\frac{1}{\lambda_{1}^{2}}\right)+\frac{\textcolor{black}{p}}{\sqrt{gH^*}}\left(\frac{1}{\lambda_{2}}-\frac{1}{\lambda_{1}}\right) \right]>0,
\end{split}
\end{equation}
where we used \eqref{subcritical}.
In particular, this means that $\bar{\eta}'(0) <\varphi'(0)$ and, as a consequence, there exists $x_{1}\in(0,x_{0})$ (recall that $x_{0}$ is defined in \eqref{dcr}) such that
\begin{equation}
    \bar{\eta}(x)<\varphi(x),\;\; \forall x\in (0,x_{1}].
\end{equation}
Assume by contradiction that there exists $x_{2}\in(x_{1},x_{0})$ such that $\bar\eta(x_{2})= \varphi(x_{2})$. Then, without loss of generality we can assume that for any $x\in[x_{1},x_{2})$, $\bar{\eta}(x)<\varphi(x)$. This implies in particular, using \eqref{positive} and \eqref{eq:comparison},
\begin{equation}
    \bar{\eta}' = \left(\frac{\delta_{1}}{\lambda_{1}}+\frac{\gamma_{2}}{\lambda_{2}}\frac{\bar{\eta}^{2}}{\varphi^{2}}\right)\varphi \leq \left(\frac{\delta_{1}}{\lambda_{1}}+\frac{\gamma_{2}}{\lambda_{2}}\right)\varphi < \left(\frac{\gamma_{1}}{\lambda_{1}}+\frac{\delta_{2}}{\lambda_{2}}\right)\varphi = \varphi'\text{ on }[x_{1},x_{2}).
\end{equation}
As a consequence, $\varphi-\bar\eta$ is increasing on $[x_{1},x_{2})$ and thus
\begin{equation}
    0=\varphi(x_{2}) - \bar{\eta}(x_{2}) >\varphi(x_{1}) - \bar{\eta}(x_{1})>0,
\end{equation}
which is in contradiction with the definition of $x_{2}$.
\end{proof}

    With Lemma \ref{lem2} and Lemma \ref{lem:cj} in hand, we are now ready to prove Theorem \ref{thm1}.
	\begin{proof}
		We construct the following Lyapunov function
		\begin{equation}
\label{eq:deflyap}
V(t)=\sum_{i=1}^{n}\int_{0}^{L_i}\left(f_{1i}(x)y_{1i}^2(t,x)+f_{2i}(x)y_{2i}^2(t,x)\right) dx,
		\end{equation}
		where
		\begin{equation}\label{def02}
				f_{1i}(x)=\alpha_i \varphi_{1i}^2\frac{1}{\lambda_{1i}\eta_i}, \quad f_{2i}(x)=\alpha_i \varphi_{2i}^2\frac{\eta_i}{\lambda_{2i}}, 
		\end{equation}
		where $\varphi_{1i}$ and $\varphi_{2i}$ are defined in \eqref{varphi12}, $\alpha_i>0$ are the weight coefficients to be chosen, and $\eta_i$ is defined as follows: $\eta_{1}$
        is the solution to 
        \begin{equation}\label{eta01}
			\left\{
			\begin{aligned}
				&\eta'_1=\left| \frac{\delta_{11}\varphi_1}{\lambda_{11}}+\frac{\gamma_{21}}{\lambda_{21}\varphi_1}\eta^2_1 \right|+\varepsilon, \\
				&\eta_1(0)=\frac{\lambda_{21}(0)}{\lambda_{11}(0)}+\varepsilon.
			\end{aligned}
			\right.
		\end{equation}
      and for $j\in\{2,3,\cdots,n\}$, $\eta_j$ is the solution to
		\begin{equation}\label{eta0j}
			\left\{
			\begin{aligned}
				&\eta'_j=\left| \frac{\delta_{1j}\varphi_j}{\lambda_{1j}}+\frac{\gamma_{2j}}{\lambda_{2j}\varphi_j}\eta^2_j \right|+\varepsilon, \\
				&\eta_j(0)=1+\varepsilon,
			\end{aligned}
			\right.
		\end{equation}
	where $\varepsilon>0$ is a constant to be chosen.
    \textcolor{black}{In particular, since $\eta_{i}$ are strictly increasing, note from \eqref{def02} that $f_{1i}\lambda_{1i}\varphi_{1i}^{-2}$ is (strictly) decreasing while $f_{2i}\lambda_{2i}\varphi_{2i}^{-2}$ is strictly increasing.
  Though the weight functions defined in \eqref{def02} seem quite implicit, in fact, we have much more information on them. Note from Lemma \ref{etaresult} and Lemma \ref{lem2} that when $\varepsilon=0$, these $\eta_i$ functions are well-defined on $[0,L_{i}]$. Indeed, from Lemma \ref{etaresult} and Remark \ref{remf}, when $\varepsilon=0$ these functions have, remarkably, an explicit closed-form expression given by \eqref{eta0i} and \eqref{explicitetai}--\eqref{mx}. This means, in particular, that the weight functions $f_{1i}$ and $f_{2i}$ also have explicit expressions when $\varepsilon=0$ given by
\begin{equation}
\label{eq:explicit-weights}
\begin{split}
f_{11}^{0} &= \alpha_{1}\frac{\varphi_{11}\varphi_{21}}{\lambda_{21}},\;\;
f_{21}^{0} = \alpha_{1}\frac{\varphi_{11}\varphi_{21}}{\lambda_{11}},\\
f_{1j}^{0} &= \alpha_{j}\frac{\varphi_{1j}\varphi_{2j}}{\lambda_{2j}m_{j}},\;\;
f_{2j}^{0} = \alpha_{j}\frac{m_{j}\varphi_{1j}\varphi_{2j}}{\lambda_{1j}},\;\;\forall j\in\{2,3,...,n\},
\end{split}
\end{equation}
where $m_{j}$ is given by \eqref{mx}.
}
    Therefore, there exists $\varepsilon_{0}>0$ such that for any $\varepsilon \in (0,\varepsilon_{0})$, $\eta_{i}$ exists on $[0,L_{i}]$ for any $i\in\{1,2,\cdots,n\}$ (see for instance \cite[Lemma 4.1-4.2]{C122} or \cite{Hartman}).
  It is straightforward to verify the existence of a constant $C >0$ depending only on the parameters of the system such that
  \begin{equation}
  \label{eq:equiv}
\frac{1}{C}\sum_{i=1}^{n}\| (h_i(t, \cdot), v_i(t, \cdot))\|^2_{L^2((0,L_i);\mathbb{R}^2)}
  \leq V(t) \leq C \sum_{i=1}^{n}\| (h_i(t, \cdot), v_i(t, \cdot))\|^2_{L^2((0,L_i);\mathbb{R}^2)}
   \end{equation}
   for any $t\in[0, +\infty)$ {\color{black} and
   \begin{equation}
       \|f_{1i}-f_{1i}^{0}\|_{L^{\infty}}\leq C\varepsilon,\;\; \|f_{2i}-f_{2i}^{0}\|_{L^{\infty}}\leq C\varepsilon,\;\;\forall i\in\{1,2,...,n\}
   \end{equation}
   which means that the weights can be made arbitrarily close to the explicit expressions \eqref{eq:explicit-weights} provided $\varepsilon$ is sufficiently small. 
   }
Similarly to \cite{HS}, the time derivative of $V$ along the trajectories of \eqref{chaform} is
		\begin{equation}
        \label{eq:computelyap}
			\begin{aligned}
				\dot V(t)=&\sum_{i=1}^{n}\int_{0}^{L_i}\left[2f_{1i}y_{1i}(-\lambda_{1i}\partial_{x}y_{1i}-\gamma_{1i}y_{1i}-\delta_{1i}
y_{2i})+2f_{2i}y_{2i}(\lambda_{2i}\partial_{x}y_{2i}-\gamma_{2i}y_{1i}-\delta_{2i}
y_{2i})\right] dx \\
				=&-B(t)-\sum_{i=1}^{n}\int_{0}^{Li} (y_{1i},y_{2i})N_i(x)(y_{1i},y_{2i})^{T}dx,
			\end{aligned}
		\end{equation}
  where
  \begin{equation}\label{BC1}
			\begin{aligned}
				B(t)=\sum_{i=1}^{n}\left[f_{1i}(L_i)\lambda_{1i}(L_i)y^2_{1i}(t,L_i)-f_{2i}(L_i)\lambda_{2i}(L_i)y^2_{2i}(t,L_i)-f_{1i}(0)\lambda_{1i}(0)y^2_{1i}(t,0)+f_{2i}(0)\lambda_{2i}(0)y^2_{2i}(t,0)\right]
			\end{aligned}
		\end{equation}
  is the term that incorporates information regarding the boundary condition and
  	\begin{equation}\label{inner01}
			\begin{aligned}
					&N_i(x)=
				\begin{pmatrix}
					-(f_{1i}\lambda_{1i})_x+2f_{1i}\gamma_{1i} & f_{1i}\delta_{1i}+f_{2i}\gamma_{2i} \\
					f_{1i}\delta_{1i}+f_{2i}\gamma_{2i}  &  (f_{2i}\lambda_{2i})_x+2f_{2i}\delta_{2i}
				\end{pmatrix}.
			\end{aligned}
		\end{equation}

	One can easily check thanks to \eqref{def02}--\eqref{eta01} (see for instance \cite{hayat2021exponential}) that
\begin{equation}\label{detN}
			\begin{aligned}
				&\det N_i>0,
			\end{aligned}
		\end{equation}
    and $N_i$ is positive definite and therefore there exists $\nu >0$ such that for any $(w_{1i},w_{2i})\in L^{2}(0,L_{i})$,
        \begin{equation}
        \label{eq:nu}
            \sum_{i=1}^{n}\int_{0}^{Li} (w_{1i},w_{2i})N_i(x)(w_{1i},w_{2i})^{T}dx \geq \nu \|(w_{1i},w_{2i})\|_{L^{2}(0,L_i)}^{2}.
        \end{equation}
Hence, the main challenge is to show that $B(t)$ is positive definite.
Denote by
  \begin{equation}\label{defZ&W}
			\begin{aligned}
				&\lambda_{1i}(x)f_{1i}(x)-\lambda_{2i}(x)f_{2i}(x)
				=:Z_i(x)=:\alpha_i\widetilde{Z}_i(x),\\
				&\lambda_{1i}(x)f_{1i}(x)+\lambda_{2i}(x)f_{2i}(x)=:W_i(x)=:\alpha_i\widetilde{W}_i(x),
			\end{aligned}
		\end{equation}

  and
  \begin{equation}
r=\begin{pmatrix} v_2(t, 0), & v_3(t, 0), & \cdots, & v_n(t, 0), & h_1(t, L_1)  \end{pmatrix}^T.
  \end{equation}
\begin{rmk}[Dependency in $\alpha_{i}$]
\label{rem:depalpha}
Note that, from \eqref{def02}, $\widetilde{Z}_i(x)$ and $\widetilde{W}_i(x)$ defined in \eqref{defZ&W} do not depend on $\alpha_{i}$. This motivates the introduction of these notations.
\end{rmk}
 Substituting boundary conditions \eqref{boul01} and \eqref{nbou01} into $B(t)$,  we obtain
	\begin{align}\label{B(t)}
				B(t)=&f_{11}(L_1)\lambda_{11}(L_1)\left( \sum_{j=2}^{n}v_j(t,0)+h_1(t,L_1)\sqrt{\frac{g}{H_1^*(L_1)}}\right)^2\nonumber\\
    &-f_{21}(L_1)\lambda_{21}(L_1)\left( \sum_{j=2}^{n}v_j(t,0)-h_1(t,L_1)\sqrt{\frac{g}{H_1^*(L_1)}}\right)^2
    \nonumber\\
    &+\sum_{j=2}^n \Bigg[-f_{1j}(0)\lambda_{1j}(0)\left(v_j(t, 0) +h_1(t, L_1)\sqrt{\frac{g}{H_j^*(0)}}\right)^2 \nonumber\\
    &+f_{2j}(0)\lambda_{2j}(0)\left(v_j(t, 0) -h_1(t, L_1)\sqrt{\frac{g}{H_j^*(0)}}\right)^2\Bigg]\nonumber\\
    &+\Bigg[-f_{11}(0)\lambda_{11}(0)\left(v_1(t, 0)-\frac{H_1^*(0)}{V_1^*(0) }\sqrt{\frac{g}{H_1^*(0)}}v_1(t, 0) \right)^2 \nonumber\\
    &+f_{21}(0)\lambda_{21}(0)\left(v_1(t, 0)+\frac{H_1^*(0)}{V_1^*(0) }\sqrt{\frac{g}{H_1^*(0)}}v_1(t, 0) \right)^2 \Bigg]\nonumber\\
    &+\sum_{j=2}^n y_{2j}^2(L_j)\left(f_{1j}(L_j)\lambda_{1j}(L_j)c_j^2-f_{2j}(L_j)\lambda_{2j}(L_j)\right)
    \nonumber\\=&r^T M r+F_1v^2_1(t,0)+\sum_{j=2}^n y_{2j}^2(L_j)\left(f_{1j}(L_j)\lambda_{1j}(L_j)c_j^2-f_{2j}(L_j)\lambda_{2j}(L_j)\right),
			\end{align}
		where
  \begin{equation}M=\begin{pmatrix}\label{MM}
      \theta_2 & Z_1(L_1) & Z_1(L_1)& \cdots & Z_1(L_1) & \omega_2 \\
      Z_1(L_1) & \theta_3 & Z_1(L_1) & \cdots & Z_1(L_1) &\omega_3\\
       Z_1(L_1) & Z_1(L_1) & \theta_4 & \cdots & Z_1(L_1) & \omega_4\\
       \vdots & \vdots & \vdots & \ddots &\vdots &\vdots\\
       Z_1(L_1) & Z_1(L_1) & Z_1(L_1) & \cdots & \theta_n & \omega_n \\
       \omega_2 & \omega_3 & \omega_4  & \cdots & \omega_n & \zeta
  \end{pmatrix} \in \mathbb{R}^{n\times n}
  \end{equation}
	with	\begin{equation}\label{theta2zeta}
			\begin{aligned}
				&	\theta _j=	Z_1(L_1)-Z_j(0),\\
				&\omega_j=\frac{\sqrt{gH_1^*(L_1)}(W_1(L_1)-W_j(0))}{H_1^*(L_1)},\\
				&\zeta=\frac{g}{H_1^*(L_1)}\left(Z_1(L_1)-\left(\sum_{j=2}^{n}Z_j(0)\right)\right),\\
				&F_1=\frac{1}{V_1^{*2}(0)}\left(\lambda_{21}(0)\lambda_{11}^2(0)f_{21}(0)-\lambda_{11}(0)\lambda_{21}^2(0)f_{11}(0)\right).\\
			\end{aligned}
		\end{equation}
Firstly, for the second term in \eqref{B(t)} which corresponds to the beginning of the channel 1, from \eqref{def02}, \eqref{eta01} and \eqref{theta2zeta}
\begin{equation}\label{defF}
  \begin{aligned}
   F_1&=\frac{\alpha_1 }{V_1^{*2}(0)} \left( \frac{\lambda^2_{11}(0)\eta^2_1(0)-\lambda^2_{21}(0)}{\eta_1(0)}\right)\\
   &=\frac{\alpha_1 }{V_1^{*2}(0)\eta_1(0)} \left( \lambda^2_{11}(0)\varepsilon^2+2\varepsilon \lambda_{21}(0)\lambda_{11}(0)\right)>0 .
   \end{aligned}
  \end{equation}
Let us look at the last terms in \eqref{B(t)} which correspond to the end of the channels $2$ to $n$. 

Note that, from \eqref{eq:hypki}, $k_{j}\sqrt{H_{j}^{*}/g}$ is located outside the roots of the polynomial
\begin{equation}\label{eta-phi01}
    X^{2}\left(1-\frac{\bar{\eta}_{j}^{2}}{\varphi_{j}^{2}}\right)+2X\left(1+\frac{\bar{\eta}_j^{2}}{\varphi_{j}^{2}}\right)+\left(1-\frac{\bar{\eta}_{j}^{2}}{\varphi_{j}^{2}}\right).
\end{equation}
From Lemma \ref{lem:cj} and the observation on $k_{j}\sqrt{H_{j}^{*}/g}$,
\begin{equation}\label{eta-phi02}
    \frac{H_{j}^{*}}{g}k_{j}^{2}\left(1-\frac{\bar{\eta}_{j}^{2}}{\varphi_{j}^{2}}\right)+2\sqrt{\frac{H_{j}^{*}}{g}}k_{j}\left(1+\frac{\bar{\eta}_j^{2}}{\varphi_{j}^{2}}\right)+\left(1-\frac{\bar{\eta}_{j}^{2}}{\varphi_{j}^{2}}\right)>0,
\end{equation}
and therefore
\begin{equation}\label{eta-phi03}
    \left(1+\sqrt{\frac{H_{j}^{*}}{g}}k_{j}\right)^{2}>\frac{\bar{\eta}_j^{2}}{\varphi_{j}^{2}}\left(1-\sqrt{\frac{H_{j}^{*}}{g}}k_{j}\right)^{2},
\end{equation}
which gives 
from \eqref{eqcj} that
\begin{equation}
\label{eq:computec}
     c_j^2>\frac{\bar\eta_j^2(L_j)}{\varphi_j^2(L_j)}.
\end{equation}
  Noticing \eqref{eta0j} and
  from the continuity of $\varepsilon\rightarrow \eta_{j}$, there exists $\varepsilon_1>0$ such that for any $\varepsilon\in(0,\varepsilon_1)$, one has for $j=2,3,\cdots,n$ that
\begin{equation}
			c_j^2>\frac{\eta_j^2(L_j)}{\varphi_j^2(L_j)}.
   \end{equation}
Since from definition \eqref{def02}
\begin{equation}\label{def002}
			\frac{\eta_j^2(L_j)}{\varphi_{j}^2(L_j)}=\frac{f_{2j}(L_j)\lambda_{2j}(L_j)}{f_{1j}(L_j)\lambda_{1j}(L_j)},
   \end{equation}
		we obtain that the terms $\sum_{j=2}^n y_{2j}^2(L_j)\left(f_{1j}(L_j)\lambda_{1j}(L_j)c_j^2-f_{2j}(L_j)\lambda_{2j}(L_j)\right)$ of \eqref{B(t)} are positive definite (with respect to $(y_{22},\cdots,y_{2j})^{T})$.\\

From Lemma \ref{etaresult}, the definition \eqref{def02}--\eqref{eta01} and \eqref{defZ&W}, and using the continuity of $\varepsilon\rightarrow \eta_1$, there exists $0<\varepsilon_{2}<\varepsilon_1$ such that for any $\varepsilon\in(0,\varepsilon_{2})$
		\begin{equation}
			Z_1(L_1)=	(\lambda_{11}f_{11}-\lambda_{21}f_{21})(L_1)=\left(\alpha_1 \varphi_{11}^2\frac{1}{\eta_1}-\alpha_1 \varphi_{21}^2\eta_1\right)(L_{1})=\left(\alpha_1 \varphi_{21}^2\left(\frac{\varphi_1^2-\eta_1^2}{\eta_1}\right)\right)(L_{1})>0.
		\end{equation}
        We are left to prove that by properly choosing the parameters $\alpha_i$, the matrix $M$ can be made positive definite.
        Noticing \eqref{defZ&W}, we now simplify the $n\times n$ matrix $M$ by selecting the specific parameters $\alpha_i$ as
  \begin{equation}
  \alpha_1=1,\quad
  \alpha_j=\frac{\widetilde{W_1}(L_1)}{\widetilde{W_j}(0)},
  \end{equation}
such that
  \begin{equation}
      W_1(L_1)=W_j(0),
  \end{equation}
which means $\omega_{j}=0$ from \eqref{theta2zeta}. Note that this is possible since the $\widetilde{W_i}$ do not depend on the $\alpha_{i}$ (see Remark \ref{rem:depalpha}). Thus, the matrix $M$ can be greatly simplified as
  \begin{equation}\label{overlineM}
    \overline M=\begin{pmatrix}
      \theta_2 & Z_1(L_1) & Z_1(L_1)& \cdots & Z_1(L_1) & 0 \\
      Z_1(L_1) & \theta_3 & Z_1(L_1) & \cdots & Z_1(L_1) &0\\
       Z_1(L_1) & Z_1(L_1) & \theta_4 & \cdots & Z_1(L_1) & 0\\
       \vdots & \vdots & \vdots & \ddots &\vdots &\vdots\\
       Z_1(L_1) & Z_1(L_1) & Z_1(L_1) & \cdots & \theta_n & 0 \\
       0& 0 & 0  & \cdots & 0& \zeta
  \end{pmatrix}.
  \end{equation}
  Then we need to analyze the potential requirements for $Z_j(0)$ \textcolor{black}{so that the matrix $\overline M$ is positive definite}.
To that end, \textcolor{black}{we focus on the $k$-th $(k=2,\cdots,n-1)$ order principal minor determinant. Thanks to the special structure of the matrix $\overline M$,
we define the matrices $\overline M_k$ $(k=2,\cdots,n-1)$ as follows
 \begin{equation}\label{overlineMk}
    \overline M_k=\begin{pmatrix}
      \theta_2 & Z_1(L_1) & Z_1(L_1)& \cdots & Z_1(L_1)  \\
      Z_1(L_1) & \theta_3 & Z_1(L_1) & \cdots & Z_1(L_1)\\
       Z_1(L_1) & Z_1(L_1) & \theta_4 & \cdots & Z_1(L_1) \\
       \vdots & \vdots & \vdots & \ddots &\vdots \\
       Z_1(L_1) & Z_1(L_1) & Z_1(L_1) & \cdots & \theta_{k+1} 
  \end{pmatrix}
  \end{equation} and make the following series of transformation for matrix $\overline M_k$ that will not change its determinant.}
\begin{itemize}
    \item The $2$-nd to the $k$-th row minus the first row respectively;
    \item Multiply the $2$-nd to the $k$-th row by $\frac{Z_1(L_1)}{Z_l(0)}$($l=3, \cdots, \textcolor{black}{k+1}$) respectively, then add these $k-1$ rows to the first row, and then restore the $2$-nd to the $k$-th rows;
    \item Multiply the $2$-nd to the $k$-th column by $\frac{Z_2(0)}{Z_l(0)}$($l=3, \cdots, \textcolor{black}{k+1}$) respectively, add these $k-1$ columns to the first columns, and then restore the $2$-nd to the $k$-th columns.
\end{itemize}
Then we obtain the transformed matrix \textcolor{black}{$\widetilde{M}_k$} as
\textcolor{black}{
 \begin{equation}
      \widetilde{M}_k=\begin{pmatrix}
     \widetilde{\theta}_2 & 0 & 0& \cdots & 0  \\
      0 & -Z_3(0) &0 & \cdots & 0 \\
      0 & 0 & -Z_4(0) & \cdots & 0 \\
       \vdots & \vdots & \vdots & \ddots &\vdots \\
       0 & 0 & 0 & \cdots & -Z_{k+1}(0) 
  \end{pmatrix},
  \end{equation}
}
where
\begin{equation}
\begin{aligned}
\widetilde{\theta}_{2}&=Z_{1}(L_{1})-Z_{2}(0)+\sum\limits_{l=3}^{\textcolor{black}{k+1}}\frac{Z_{2}(0)Z_{1}(L_{1})}{Z_{l}(0)},
\end{aligned}
\end{equation}

Next, we emphasize that all of $Z_j(0)$ \textcolor{black}{$(j=2,\cdots,k+1)$ being  negative is a sufficient condition for the determinant of the diagonal matrix $\widetilde{M}_k$ $(k=2,\cdots,n-1)$ being positive}. Indeed, it is easy to check that $\widetilde{\theta}_2$ is positive under $Z_j(0)<0$. 	
Moreover, from \eqref{def02} and \eqref{defZ&W}, one has 
\begin{equation}\label{equiv}
Z_j(0)<0\Longleftrightarrow  \eta_j(0)>1,
\end{equation}
which together with the equations \eqref{eta0j} satisfied by $\eta_j$, we get immediately that \textcolor{black}{the determinant of the diagonal matrix $\widetilde{M}_k$ $(k=2,\cdots,n-1)$ is positive.  Moreover, from \eqref{theta2zeta}, $\theta_2$ and $\zeta$ are also positive (using that $Z_j(0)<0$), which together with the positive definiteness of $\widetilde M_k$ shows that $\overline M$ is positive definite. As a consequence,} $B(t)$ is positive definite and therefore, combining this with \eqref{eq:nu} and \eqref{eq:equiv}, there exists $\mu>0$ such that

\begin{equation}
    \dot V(t) \leq -\mu V(t),
\end{equation}
which ends the proof of Theorem \ref{thm1}.
	\end{proof}
\begin{rmk}\label{cr1}
We point out that though it would seem \emph{a priori} that condition \eqref{eq:hypki} required for the control parameters depends on the value of $(H^*_j,V^*_j)$ for all $x\in[0,L_j]$, thanks to the key observation of 
Remark \ref{remf}, it only depends on the values of $H^*_j$ at two ends $x=0$ and $x=L_j$. More precisely, condition \eqref{eq:hypki} is equivalent to condition \eqref{paraphy} defined
in Theorem \ref{thm0} for the nonlinear system.
Indeed, one just need to substitute the explicit expression of $\bar\eta_j$ in \eqref{explicitetai} to the bounds of condition \eqref{eq:hypki}, direct computations will lead to the result.
\end{rmk}

 \section{A Lyapunov Function for the Original Nonlinear System of the Star-shaped Model}
 \label{sec:nonlinear} 
We use the following change of variables
\begin{equation}
\label{eq:nlchangevar}
y_{1i} = V_i-V^*_i + 2\sqrt{g H_i}-2\sqrt{g H_{i}^{*}},\;\; y_{2i} = V_i-V^*_i - 2\sqrt{g H_i}+2\sqrt{g H_{i}^{*}}.
\end{equation}
Note that the linearized version of this change of variable is exactly \eqref{RiemannIn} 
with the definition of the disturbances \eqref{disturbance}.
Denote by $\textbf{y}_{i}= (y_{1i},y_{2i})^{T}$, the nonlinear system \eqref{sys01} becomes
\begin{equation}
	\label{eq:sysdiag}
	\begin{split}
		\partial_{t}\textbf{y}_{i} + \Lambda_{i}(\textbf{y}_{i},x)\partial_{x}\textbf{y}_{i} +B_i(\textbf{y}_i,x)=0,
	\end{split}
\end{equation}
where the diagonal matrices $\Lambda_{i}(\textbf{y}_{i},x)$ are
\begin{equation}
\Lambda_{i}(\textbf{y}_i,x)=\begin{pmatrix}
\Lambda_{1i}(\textbf{y}_i,x)&0\\0&\Lambda_{2i}(\textbf{y}_i,x)
\end{pmatrix}
\end{equation}
with
\begin{equation}
	\label{eq:lambda}
	\begin{split}
		\Lambda_{1i}(\mathbf{0},x)=\lambda_{1i}(x), \\
		\Lambda_{2i}(\mathbf{0},x)= \lambda_{2i}(x),
	\end{split}
\end{equation}
and $B_i(\textbf{y}_i,x)$ satisfies
\begin{equation}
B_i(\textbf{0},x)=\textbf{0},\quad \frac{\partial B_i}{\partial \textbf{y}_i}(\textbf{0},x)=\begin{pmatrix}\gamma_{1i}(x)&\delta_{1i}(x)\\
\gamma_{2i}(x)&\delta_{2i}(x)\end{pmatrix}.
\end{equation}
The boundary condition \eqref{bou01} can be rewritten as 
\begin{equation}\label{nbc}
		\mathbf{y}^{in}(t) 
	=\mathcal{H}(\mathbf{y}^{out}(t)),
\end{equation}
where
\begin{equation}
		\begin{split}
			&\mathbf{y}^{in}(t) =(y_{11}(t,0),y_{21}(t,L_1),y_{12}(t,0),y_{22}(t,L_2),\cdots,y_{1n}(t,0),y_{2n}(t,L_n)),\\
			&\mathbf{y}^{out}(t) = (y_{11}(t,L_1),y_{21}(t,0),y_{12}(t,L_2),y_{22}(t,0),\cdots,y_{1n}(t,L_n),y_{2n}(t,0)),
		\end{split}
\end{equation}
and 
\begin{equation}
	\mathcal{H}(\textbf{0})=\textbf{0}.
\end{equation}
Since we use a combination of Riemman invariants and physical variables in the analysis of the Lyapunov function, we will not give the explicit expression $\mathcal{H}$ here for the sake of simplicity, which follows directly from the boundary conditions \eqref{bou01}, \eqref{bou01s} and the change of variables \eqref{eq:nlchangevar}.
Note that, despite the nonlinear change of variables we still have an equivalence in the $H^{2}$-norm between these variables and the physical variables around the equilibrium, i.e.
\begin{lem}
\label{lem:eqnorms}
There exist positive constants $\bar\delta$, $\bar C$ such that for every 
$$\|(H_{i}(t,\cdot)-H^{*}_{i},V_{i}(t,\cdot)-V^{*}_{i})\|_{H^{2}((0,L_{i});\mathbb{R}^{2})}<\bar\delta,$$
we have
\begin{align}\label{equivnd}
  \frac{1}{\bar C}\|(H_{i}(t,\cdot)-H^{*}_{i},V_{i}(t,\cdot)-V^{*}_{i})\|_{H^{2}((0,L_{i});\mathbb{R}^{2})}\leq &\|(y_{1i}(t,\cdot),y_{2i}(t,\cdot))\|_{H^{2}((0,L_{i});\mathbb{R}^{2})}\nonumber\\
  \leq &\bar C\|(H_{i}(t,\cdot)-H^{*}_{i},V_{i}(t,\cdot)-V^{*}_{i})\|_{H^{2}((0,L_{i});\mathbb{R}^{2})}.
\end{align}
\end{lem}
Inequality \eqref{equivnd} follows directly from the definition \eqref{eq:nlchangevar} and straightforward computations, we omit the details here.

We consider an augmented system with state $(\textbf{y}_i,\partial_t\textbf{y}_i,\partial_{tt}\textbf{y}_i)$ where the dynamics
of $\partial_t\textbf{y}_i$ and $\partial_{tt}\textbf{y}_i$ are simply obtained by taking \textcolor{black}{partial derivatives with respect to time} of the
equations \eqref{eq:sysdiag} and the boundary conditions \eqref{nbc}. Very similar to dealing with the linearized system in Section \ref{sec:lyap}, we use the following extended Lyapunov function 
\begin{equation}
	\label{eq:defV}
	\begin{split}
		V =&\sum\limits_{k=0}^{2} \sum\limits_{i=1}^{n}\int_{0}^{L_{i}}f_{1i}(x)(\partial_{t}^{k}y_{1i}(t,x))^{2}+f_{2i}(x)(\partial_{t}^{k}y_{2i}(t,x))^{2}dx
	\end{split}
\end{equation}
to prove the exponential asymptotic stability of the nonlinear system:
\begin{thm}\label{thm2}
 Under condition 
 \eqref{paraphy},
	the nonlinear system \eqref{eq:sysdiag} and \eqref{nbc} is (locally) exponentially stable for the $H^2$-norm, i.e. there exist $\tilde\delta>0$, $\nu>0$ and $\tilde C>0$ such that, for every initial condition  $(y_{1i}(0,x),y_{2i}(0,x)) \in H^2((0,L_i);\mathbb{R}^2) $, $i\in\{1,\cdots,n\}$ satisfying  $\| (y_{1i}(0,x),y_{2i}(0,x))\|_{H^2((0,L_i);\mathbb{R}^2)} < \tilde\delta $ and the first-order compatibility conditions associated to \eqref{eq:sysdiag}, \eqref{nbc} (see \cite{BCBook}), there exists a unique solution $(y_{1i},y_{2i})\in C^0([0,+\infty);H^2((0,L_i);\mathbb{R}^2))$ to the Cauchy problem \eqref{eq:sysdiag} and \eqref{nbc} satisfying
	\begin{equation}\label{esty}
		\sum_{i=1}^n\|(y_{1i}(t, \cdot ),y_{2i}(t,\cdot)) \|_{H^2((0,L_i);\mathbb{R}^2)}\leq\tilde C e^{-\nu t}\left(\sum_{i=1}^n\| (y_{1i}(0,x),y_{2i}(0,x))\|_{H^2((0,L_i);\mathbb{R}^2)} \right),\
		\forall t\in[0,+\infty).
	\end{equation}
\end{thm}
\begin{proof}
For any $T>0$, there exists $\delta(T)$ such that for any initial condition
$\| (y_{1i}(0,x),y_{2i}(0,x))\|_{H^2((0,L_i);\mathbb{R}^2)} < \delta(T) $ the well-posedness is guaranteed by \cite[Theorem B.1]{BCBook} of which the proof is an adaption of \cite{Kato} and \cite{LiYu}. We now prove the estimation \eqref{esty} on $[0,T]$.
As usual, we temporarily assume that the solutions $(y_{1i},y_{2i})$ are of class $C^3$ in space (this is possible thanks to the well-posedness of the system in $C^{3}$, see for instance \cite{C1}). Differentiating $V$ along the smooth solutions of \eqref{eq:sysdiag}, then we divide the result into boundary part and internal part:
\begin{equation}\label{DVkt}
	\frac{d	V }{dt} \leq  -\widetilde{B}(t)-I(t),
	\end{equation}
\textcolor{black}{where $\widetilde{B}(t)$ contains the information on the boundaries and $I(t)$ is the interior term. Compared with the stabilization of the linearized system, there are some higher order terms in these two terms. These terms can be bounded using the following Sobolev inequality, 
   \begin{equation}\label{soe}
\|\textbf{y}_i(t,\cdot)\|_{C^1([0,L_i];\mathbb{R}^2)}\leq C \|\textbf{y}_i(t,\cdot)\|_{H^{2}((0,L_i);\mathbb{R}^2)},
   \end{equation}
  where $C>0$ is independent of $T$ and $\textbf{y}$. As a consequence the expression of $I(t)$ is 
   \begin{equation}
	\label{eq:exprI}
	\begin{split}
		I(t) = \sum\limits_{k=0}^{2}\sum\limits_{i=1}^{n}&\int_{0}^{L_{i}}\partial_{t}^{k}\textbf{y}_{i}^{T}N_{i}(x)\partial_{t}^{k}\textbf{y}_{i} dx+O\left(\sum\limits_{i=1}^{n}\|\textbf{y}_{i}(t,\cdot)\|^3_{H^2((0,L_i);\mathbb{R}^2)}\right).
        \end{split}
\end{equation}
In the above expression, $N_i(x)$ is still defined by \eqref{inner01} and is positive definite. Here and hereafter, $O(x)$ refers to a function such that $O(x)/|x|$ is bounded when $|x|\rightarrow 0$. 
A similar estimation for the higher order terms of \eqref{eq:exprI} can be found in \cite[Section 6.2]{BCBook} with details.
We next estimate the term }
\begin{equation}
	\label{eq:exprB}
	\begin{split}
		\widetilde{B}(t) =\sum\limits_{k=0}^{2}\sum\limits_{i=1}^{n}&\left[\Lambda_{1i}(\textbf{y}_{i},x)f_{1i}(x)(\partial_{t}^{k}y_{1i})^{2}-\Lambda_{2i}(\textbf{y}_{i},x)f_{2i}(x)(\partial_{t}^{k}y_{2i})^{2}\right]_{0}^{L_{i}}.
	\end{split}
\end{equation}

For the sake of simplicity, we drop the argument $t$ in the following computations. We decompose $\widetilde{B}$ as $\widetilde{B} = \widetilde{B}_{M}+\widetilde{B}_{S}$, where $\widetilde{B}_{M}$ corresponds to the boundary terms involved at the multiple nodes, and $\widetilde{B}_{S}$ at the simple nodes where either a control or an imposed flux applies. \textcolor{black}{In order to compare with the analysis for the linearized system in Section \ref{sec:lyap}, we rewrite these terms as a combination of Riemann invariants and physical variables with higher order terms expressed in Riemann invariants,} namely

\begin{equation}
	\begin{aligned}\label{nnu0}
		\widetilde{B}_{M}(t)= &  \sum\limits_{k=0}^{2}\sum\limits_{j=2}^{n}\left[-\Lambda_{1j}(\textbf{y}_{j}(0), 0)f_{1j}(0)(\partial_{t}^{k}y_{1j}(0))^{2}+\Lambda_{2j}(\textbf{y}_{j}(0),0)f_{2j}(0)(\partial_{t}^{k}y_{2j}(0))^{2}\right]    \\
		&+\sum_{k=0}^{2}\left[\Lambda_{11}(\textbf{y}_{1}(L_1),L_1)f_{11}(L_1)(\partial_{t}^{k}y_{11}(L_1))^{2}-\Lambda_{21}\textbf{y}_{1}(L_1),L_1)f_{21}(L_1)(\partial_{t}^{k}y_{21}(L_1))^{2}\right]  \\	=&\sum_{k=0}^2\partial_{t}^kr^TM(x)\partial_{t}^kr+O\left(\sum\limits_{k=0}^{2}\sum\limits_{i=1}^{n}\left(|\partial_{t}^{k}\textbf{y}_{i}(0)|^{2}+|\partial_{t}^{k}\textbf{y}_{1}(L_1)|^{2}\right)\|\textbf{y}_{i}(t,\cdot)\|_{C^0([0,L_i];\mathbb{R}^2)}\right),
        \end{aligned}
        \end{equation}
       \begin{equation}\label{nnu1}
        \begin{aligned}
		\widetilde{B}_{S}(t)=&  \sum\limits_{k=0}^{2}\sum\limits_{j=2}^{n}\left[\Lambda_{1j}(\textbf{y}_{j}(L_j),L_j)f_{1j}(L_j)(\partial_{t}^{k}y_{1j}(L_j))^{2}-\Lambda_{2j}(\textbf{y}_{j}(L_j),L_j)f_{2j}(L_j)(\partial_{t}^{k}y_{2j}(L_j))^{2}\right]    \\
		&+\sum_{k=0}^{2}\left[-\Lambda_{11}(\textbf{y}_{1}(0),0)f_{11}(0)(\partial_{t}^{k}y_{11}(0))^{2}+\Lambda_{21}(\textbf{y}_{1}(0),0)f_{21}(0)(\partial_{t}^{k}y_{21}(0))^{2}\right]  \\
=&F_1\left(\partial_{t}^{k}v_1(0)\right)^2+\sum_{j=2}^n y_{2j}^2(L_j)\left(f_{1j}(L_j)\lambda_{1j}(L_j)c_j^2-f_{2j}(L_j)\lambda_{2j}(L_j)\right)\\
&+O\left(\sum\limits_{k=0}^{2}\sum\limits_{i=1}^{n}\left(|\partial_{t}^{k}\textbf{y}_{i}(L_i)|^{2}+|\partial_{t}^{k}\textbf{y}_{1}(0))|^{2}\right)\|\textbf{y}_i(t,\cdot)\|_{C^0([0,L_i];\mathbb{R}^2)}\right),\\
		\end{aligned}
	\end{equation}
	where the definitions of the vector $r$, the matrix $M$ and $F_1$ as well as the parameters $c_j$ are the same as those in Section \ref{sec:lyap} for the linearized system. \textcolor{black}{The estimations of \eqref{nnu0}-\eqref{nnu1} are obtained using \eqref{eq:nlchangevar}.} Thus, \eqref{DVkt} becomes  
    \begin{equation}\label{DVt2}
    \begin{aligned}
    	\frac{d	V}{dt} \leq&-\sum\limits_{k=0}^{2}\sum\limits_{i=1}^{n}\int_{0}^{L_{i}}\partial_{t}^{k}\textbf{y}_{i}^{T}N_{i}(x)\partial_{t}^{k}\textbf{y}_{i} dx+O\left(\sum\limits_{i=1}^{n}\|\textbf{y}_{i}(t,\cdot)\|^3_{H^2((0,L_i);\mathbb{R}^2)}\right)\\
        &-\sum_{k=0}^2\partial_{t}^kr^TM(x)\partial_{t}^kr-F_1\left(\partial_{t}^{k}v_1(0)\right)^2-\sum_{j=2}^n y_{2j}^2(L_j)\left(f_{1j}(L_j)\lambda_{1j}(L_j)c_j^2-f_{2j}(L_j)\lambda_{2j}(L_j)\right)\\
&+O\left(\sum\limits_{k=0}^{2}\sum\limits_{i=1}^{n}\left(|\partial_{t}^{k}\textbf{y}_{i}(0)|^{2}+|\partial_{t}^{k}\textbf{y}_{i}(L_i))|^{2}\right)\|\textbf{y}_i(t,\cdot)\|_{C^0([0,L_i];\mathbb{R}^2)}\right).
        \end{aligned}
    \end{equation}
   We recover the quadratic formula of the linear case augmented with (at least) cubic terms. 
   \textcolor{black}{Using \eqref{eq:nlchangevar}} and Sobolev inequality \eqref{soe} again, 
   {\color{black}
   for sufficiently small $\|\textbf{y}_i(t,\cdot)\|_{H^2((0,L_i);\mathbb{R}^2)}$, \begin{equation}
\begin{aligned}
\label{eq:getlinbound}
&-\sum_{k=0}^2\partial_{t}^kr^TM(x)\partial_{t}^kr-F_1\left(\partial_{t}^{k}v_1(0)\right)^2-\sum_{j=2}^n y_{2j}^2(L_j)\left(f_{1j}(L_j)\lambda_{1j}(L_j)c_j^2-f_{2j}(L_j)\lambda_{2j}(L_j)\right)\\
&+O\left(\sum\limits_{k=0}^{2}\sum\limits_{i=1}^{n}\left(|\partial_{t}^{k}\textbf{y}_{i}(0)|^{2}+|\partial_{t}^{k}\textbf{y}_{i}(L_i))|^{2}\right)\|\textbf{y}_i(t,\cdot)\|_{C^0([0,L_i];\mathbb{R}^2)}\right)\\
\leq& 
-\sum_{k=0}^2\partial_{t}^kr^T\left[M(x)+O(\|\textbf{y}_i(t,\cdot)\|_{H^2((0,L_i);\mathbb{R}^2)})\right]\partial_{t}^kr-\left[F_1+O(\|\textbf{y}_i(t,\cdot)\|_{H^2((0,L_i);\mathbb{R}^2)})\right]\left(\partial_{t}^{k}v_1(0)\right)^2\\
&-\sum_{j=2}^n y_{2j}^2(L_j)\left[f_{1j}(L_j)\lambda_{1j}(L_j)c_j^2-f_{2j}(L_j)\lambda_{2j}(L_j)+O(\|\textbf{y}_i(t,\cdot)\|_{H^2((0,L_i);\mathbb{R}^2)})\right].
        \end{aligned}
   \end{equation}}
    Using the well-posedness of the system, \textcolor{black}{for any $\delta_{0}>0$ and $T>0$,} there \textcolor{black}{exists} 
    $\delta(T)$    
   such that if $\| (y_{1i}(0,x),y_{2i}(0,x))\|_{H^2((0,L_i);\mathbb{R}^2)} < \delta(T) $ then $\|\textbf{y}_i(t,\cdot)\|_{H^2((0,L_i);\mathbb{R}^2)}<\delta_{0}$.
   \textcolor{black}{Combining this with \eqref{DVt2} and \eqref{eq:getlinbound},
    under condition \eqref{paraphy} (which is exactly condition \eqref{eq:hypki} by Remark \ref{cr1}), and using
   the positive definiteness of $N$ and $M$ and the positivity of $F_{1}$ and of
   \begin{equation*}
       \left(f_{1j}(L_j)\lambda_{1j}(L_j)c_j^2-f_{2j}(L_j)\lambda_{2j}(L_j)\right),
   \end{equation*}
   there exists $\nu>0$ independent of $T>0$ and $\delta(T)$ such that if $\| (y_{1i}(0,x),y_{2i}(0,x))\|_{H^2((0,L_i);\mathbb{R}^2)} < \delta(T)$, then
}
\begin{equation}\label{ndv}
	\frac{d V}{dt}\leq -\nu V,\;\text{ for any }t\in[0,T].
\end{equation}
Using a density argument similar to \cite[Comment 4.6]{BCBook},
\eqref{ndv} still holds in
the distribution sense $\textbf{y}_i\in C^{0}([0,T];H^2((0,L_i);\mathbb{R}^2))$.
Finally, since $T$ is chosen arbitrarily, one can choose $T_{1}$ sufficiently large and proceed by induction to show that there exists $\tilde\delta$ independent of $T$ such that Theorem \ref{thm2} holds (see \cite[Proof of Theorem 4.1]{bastin2019exponential} for more details).
\end{proof}
\begin{rmk}
Theorem \ref{thm0} then follows from Theorem \ref{thm2} together with the equivalence of the norms given by Lemma \ref{lem:eqnorms}.
\end{rmk}
\section{A Lyapunov Function for the Linearized System of the Tree-shaped Network}\label{LtreeL}
    
In this section, we show how to extend the proof of Theorems \ref{thm1}, \ref{thm0} to the tree-shaped network. 
We only give the sketch of the proof for the linearized system that captures the main spirit of the approach. The adaptation to the stabilization for the nonlinear system is identical to the star-shaped model, so we omit it here. After linearization around any given steady-states $(H^*_i(x),V^*_i(x))$ and introducing the same Riemman invariants as in \eqref{RiemannIn}, the equations \eqref{chaform}--\eqref{all>0} and \eqref{positive} still holds for all $i\in\{1,\cdots,n\}$. \\
The difference lies in the boundary conditions, which \eqref{treebound} become
\begin{equation}
	\begin{aligned}\label{ltreeboul01}
			&	A: h_1(t,0)=-\frac{H_1^*(0)}{V_1^*(0)}v_1(t,0),\\
			& J^i_M:
   v_i(t,L_i)=\sum_{j\in \mathcal{I}^i_{out}}v_j(t,0), \quad i\in\mathcal{M},\\
			& \ \ \ \  \ h_i(t,L_i)=h_j(t,0),\quad i\in\mathcal{M},\quad j\in \mathcal{I}^i_{out},\\
			& J^i_S:y_{1i}(t,L_i)=c_iy_{2i}(t,L_i),\quad i\in\mathcal{S}, \\
		\end{aligned}
	\end{equation}
    \begin{equation}\label{treeeqci}
	c_i=\frac{1+k_i\sqrt{\frac{H_i^*(L_i)}{g}}}{k_i\sqrt{\frac{H_i^*(L_i)}{g}}-1},\quad i\in\mathcal{S}.  \ \
		\end{equation}
Here, as in star-shaped model, we keep the expression of the physical boundary conditions for the starting node $A$ and the multiple nodes $J^i_M$ for simplicity, rather than re-expressing them in Riemann invariants.
		
        We construct the following Lyapunov function
		\begin{equation}\label{treeL}
			V(t)=\sum_{i=1}^{n}\int_{0}^{L_i}\left(f_{1i}(x)y_{1i}^2(t,x)+f_{2i}(x)y_{2i}^2(t,x)\right) dx,
		\end{equation}
		where $f_{1i}$, $f_{2i}$, $\varphi_{1i}$ and $\varphi_{2i}$ are still defined in \eqref{def02} and \eqref{varphi12} respectively, $\alpha_i>0$ are the weight coefficients to be chosen, $\eta_1$ is defined as the solution to \eqref{eta01} and $\eta_j$ $(j=2,\cdots,n)$ is defined as the solution to \eqref{eta0j} with sufficiently small $\varepsilon$. The time derivative of $V$ along the trajectories of \eqref{chaform} is
		\begin{equation}
			\begin{aligned}
				\dot V(t)=&\sum_{i=1}^{n}\int_{0}^{L_i}\left[2f_{1i}y_{1i}(-\lambda_{1i}\partial_{x}y_{1i}-\gamma_{1i}y_{1i}-\delta_{1i}
y_{2i})+2f_{2i}y_{2i}(\lambda_{2i}\partial_{x}y_{2i}-\gamma_{2i}y_{1i}-\delta_{2i}
y_{2i})\right] dx \\
				=&-B(t)-\sum_{i=1}^{n}\int_{0}^{Li} (y_{1i},y_{2i})N_i(t, x)(y_{1i},y_{2i})^{T}dx,
			\end{aligned}
		\end{equation}
  where $B(t)$ and $N(t)$ are given in \eqref{BC1} and \eqref{inner01} respectively.
In order to deal with the boundary term $B(t)$ for the tree-shaped network, we separate \eqref{BC1} according to each local star-shaped model using \eqref{ltreeboul01}. 
	\begin{equation}\label{treeB(t)0}
			\begin{aligned}
				B(t)=&\sum_{i\in \mathcal{M}}\textcolor{black}{S_{1,i}}
    +\sum_{\textcolor{black}{i}\in \mathcal{M}}\textcolor{black}{S_{2,\textcolor{black}{i}}}
    +\sum_{i\in \mathcal{S}}\textcolor{black}{S_{2,i}}
    +\sum_{i\in \mathcal{S}} \textcolor{black}{S_{3,i}},
			\end{aligned}
		\end{equation}
where $S_{1,i}$ represents the contribution of the boundary terms at $x=L_i$ for each channel $i\in\mathcal{M}$, $S_{2,i}$ represents the contribution of the boundary terms at $x=0$ for each channel $i\in \mathcal{M}$ or $i\in\mathcal{S}$, and $S_{3,i}$ represents the contribution of the boundary terms at $x=L_i$ for each channel $i\in\mathcal{S}$ that are given as follows
\textcolor{black}{
\begin{equation}
\begin{split}
S_{1,i} =& f_{1i}(L_i)\lambda_{1i}(L_i)\left( \sum_{j\in\mathcal{I}^i_{out}}v_j(t,0)+h_i(t,L_i)\sqrt{\frac{g}{H_i^*(L_i)}}\right)^2\\
    &-f_{2i}(L_i)\lambda_{2i}(L_i)\left( \sum_{j\in\mathcal{I}^i_{out}}v_j(t,0)-h_i(t,L_i)\sqrt{\frac{g}{H_i^*(L_i)}}\right)^2,\\
    S_{2,\textcolor{black}{i}} =& -f_{1\textcolor{black}{i}}(0)\lambda_{1\textcolor{black}{i}}(0)\left(v_{\textcolor{black}{i}}(t,0)+h_{\textcolor{black}{i}}(t,0)\sqrt{\frac{g}{H_{\textcolor{black}{i}}^*(0)}}\right)^2\\
    &+f_{2\textcolor{black}{i}}(0)\lambda_{2\textcolor{black}{i}}(0)\left(v_{\textcolor{black}{i}}(t,0)-h_{\textcolor{black}{i}}(t,0)\sqrt{\frac{g}{H_{\textcolor{black}{i}}^*(0)}}\right)^2,\\
    S_{3,i} =&\  y_{2i}^2(L_i)\left(f_{1i}(L_i)\lambda_{1i}(L_i)c_i^2-f_{2i}(L_i)\lambda_{2i}(L_i)\right).
\end{split}
\end{equation}
}
Note that, from the definition of the tree-shaped network given in Section \ref{ssec:treeshape}, for any channel $j\in\mathcal{M}\setminus\{1\}$ and any $j\in\mathcal{S}$ there exists a multiple node at the origin of this channel, and thus a unique $i\in \mathcal{M}$ such that $j\in \mathcal{I}_{out}^{i}$. Reciprocally, the branches in $\mathcal{M}\setminus\{1\}$ and $\mathcal{S}$ represent exactly all the branches that are outgoing from all multiple nodes. More precisely, 
\begin{equation}
 \mathcal{M}\setminus\{1\}\cup \mathcal{S} = \bigcup\limits_{i\in\mathcal{M}}\mathcal{I}_{out}^{i},
\end{equation}
together with
\begin{equation}
    \mathcal{M}\cap \mathcal{S} = \emptyset,\quad \mathcal{I}_{out}^{i}\cap\mathcal{I}_{out}^{j} = \emptyset,\; \forall j\neq i.
\end{equation}
As a consequence
	\begin{equation}
    \label{eq:pretreeB}
			\begin{aligned}
    &\sum_{j\in \mathcal{M}\setminus\{1\}} \textcolor{black}{S_{2,j}}+\sum_{\textcolor{black}{j}\in \mathcal{S}}\textcolor{black}{S_{2,\textcolor{black}{j}}}=\sum_{i\in \mathcal{M}}\sum\limits_{j\in\mathcal{I}_{out}^{i}} \textcolor{black}{S_{2,j}},
			\end{aligned}
		\end{equation}
and therefore, \textcolor{black}{using condition \eqref{ltreeboul01} at node $A$,} \eqref{treeB(t)0} becomes
\begin{equation}\label{treeB(t)01}
			\begin{aligned}
				B(t)=&\sum_{i\in\mathcal{M}} \textcolor{black}{S_{1,i}}+\sum_{i\in\mathcal{M}}\sum_{j\in\mathcal{I}^i_{out}} \textcolor{black}{S_{2,j}}+\sum_{i\in \mathcal{S}} \textcolor{black}{S_{3,i}}\\
    &+\Bigg[-f_{11}(0)\lambda_{11}(0)\left(v_1(t, 0)-\frac{H_1^*(0)}{V_1^*(0) }\sqrt{\frac{g}{H_1^*(0)}}v_1(t, 0) \right)^2 \\
    &+f_{21}(0)\lambda_{21}(0)\left(v_1(t, 0)+\frac{H_1^*(0)}{V_1^*(0) }\sqrt{\frac{g}{H_1^*(0)}}v_1(t, 0) \right)^2 \Bigg].
			\end{aligned}
		\end{equation}
We denote 
\begin{equation}
r_i=\begin{pmatrix} v_{i_2}(t, 0), &v_{i_3}(t, 0),&\cdots, &v_{i_{\deg(J^i_M)}}(t, 0),& h_i(t, L_i)  \end{pmatrix}^T,
  \end{equation}
  where here and hereafter, we denote the indices in the set $\mathcal{I}^i_{out}$ as $i_2,i_3,\cdots,i_{\deg(J^i_M)}$. For any $i\in\mathcal{M}$, we also define as previously $Z_{i}$ and $W_{i}$ by \eqref{defZ&W}, $F_{1}$ as in \eqref{theta2zeta} and
\begin{equation}M_i=\begin{pmatrix}
      \theta_{i_2} & Z_i(L_i) & Z_i(L_i)& \cdots & Z_i(L_i) & \omega_{i_2} \\
      Z_i(L_i) & \theta_{i_3} & Z_i(L_i) & \cdots & Z_i(L_i) &\omega_{i_3}\\
       Z_i(L_i) & Z_i(L_i) & \theta_{i_4} & \cdots & Z_i(L_i) & \omega_{i_4}\\
       \vdots & \vdots & \vdots & \ddots &\vdots &\vdots\\
       Z_i(L_i) & Z_i(L_i)& Z_i(L_i) & \cdots & \theta_{i_{\deg(J^i_M)}} & \omega_{i_{\deg(J^i_M)}} \\
       \omega_{i_2} & \omega_{i_3} & \omega_{i_4}  & \cdots & \omega_{i_{\deg(J^i_M)}} & \zeta_i
  \end{pmatrix} \in\mathbb{R}^{\deg(J^i_M)\times\deg(J^i_M)}
  \end{equation}
	with	\begin{equation}\label{treetheta2zeta}
			\begin{aligned}
				&	\theta_{i_l}=	Z_i(L_i)-Z_{i_l}(0),\quad l=2,3,\cdots,\deg(J^i_M),\\
				&\omega_{i_l}=\frac{\sqrt{gH_i^*(L_i)}(W_i(L_i)-W_{i_l}(0))}{H_i^*(L_i)}, \quad l=2,3,\cdots,\deg(J^i_M),\\
				&\zeta_i=\frac{g}{H_i^*(L_i)}\left(Z_i(L_i)-\left(\sum_{l=2}^{\deg(J^i_M)}Z_{i_l}(0)\right)\right).
			\end{aligned}
		\end{equation}
    With these notations,
\eqref{treeB(t)01} becomes:
\begin{equation}\label{treeB(t)}
    B(t) =\sum_{i\in\mathcal{M}}r_{i}^T M_i r_i+F_1v^2_1(t,0)+\sum_{i\in\mathcal{S}} y_{2i}^2(L_i)\left(f_{1i}(L_i)\lambda_{1i}(L_i)c_i^2-f_{2i}(L_i)\lambda_{2i}(L_i)\right).
\end{equation}
    
The key point is to prove the positiveness of the matrices $M_i$ ($i\in\mathcal{M}$) and this can be done exactly as in Section \ref{sec:lyap}:
choosing $\alpha_k$, $k=1,2,\cdots,n$ as follows
\begin{equation}\label{alphai}
  \alpha_i=1,\quad
  \alpha_{i_l}=\frac{\widetilde{W_i}(L_i)}{\widetilde{W_{i_l}}(0)},\quad i\in\mathcal{M},\,  l=2,3,\cdots,\deg(J^i_M)
  \end{equation}
we obtain as previously 
$\omega_{i_l}=0$, $l=2,3,\cdots,\deg(J^i_{M})$ from \eqref{treetheta2zeta}.
  
 Thus, the matrix $M_i$ can be simplified as
 \begin{equation}\overline M_i=\begin{pmatrix}
      \theta_{i2} & Z_i(L_i) & Z_i(L_i)& \cdots & Z_i(L_i) & 0 \\
       Z_i(L_i) & \theta_{i3} & Z_i(L_i) & \cdots & Z_i(L_i) &0\\
        Z_i(L_i) & Z_i(L_i) & \theta_{i4} & \cdots & Z_i(L_i) & 0\\
        \vdots & \vdots & \vdots & \ddots &\vdots &\vdots\\
        Z_i(L_i) & Z_i(L_i)& Z_i(L_i) & \cdots & \theta_{i\deg(J^i_M)} & 0 \\
       0 & 0 &0  & \cdots & 0 & \zeta_i
   \end{pmatrix}
  \end{equation}
  
  Since the structure of each matrix $\overline M_i$ is the same as the one of matrix $\overline M$ defined in \eqref{overlineM}, if we denote by $\overline M_{ik}$, where $k\in\{2,\cdots,\deg(J^i_M)-1\}$, the matrix corresponding to the $k$-th order principal minor determinant of $\overline M_{i}$, using the same transformation for matrix $\overline M_{ik}$. We obtain 
 the transformed matrix $\widetilde{M}_{ik}$ as
  \begin{equation}
      \widetilde{M}_{ik}=\begin{pmatrix}
     \widetilde{\theta}_{i_2} & 0 & 0& \cdots & 0  \\
      0 & -Z_{i_3}(0) &0 & \cdots & 0 \\
      0 & 0 & -Z_{i_4}(0) & \cdots & 0 \\
       \vdots & \vdots & \vdots & \ddots &\vdots \\
       0 & 0 & 0 & \cdots & -Z_{i_{k+1}}(0)  
  \end{pmatrix},
  \end{equation}

where
\begin{equation}\label{trthe}
\begin{aligned}
\widetilde{\theta}_{i_2}&=Z_{i}(L_{i})-Z_{i_2}(0)+\sum\limits_{l=3}^{\deg(J^i_M)}\frac{Z_{i_2}(0)Z_{i}(L_{i})}{Z_{i_l}(0)},
\end{aligned}
\end{equation}
Now, \textcolor{black}{we would like to show} that for each $i\in\mathcal{M}$, $Z_{i}(L_{i})>0$, just as we showed that $Z_{1}(L_{1})>0$ for the star-shaped case. However, this is not \emph{a priori} obvious since $\eta_{1}(0)=\lambda_{21}/\lambda_{11}+\varepsilon<1$ while $\eta_{i}(0)=1+\varepsilon>1$ so it is not \emph{a priori} clear that $\eta_{i}$ remains sufficiently small to ensure that $Z_{i}(L_{i})>0$. The crucial point to conclude lies in Lemma \ref{lem:cj}. Indeed,
\begin{align}\label{trzp}
			Z_i(L_i)=&	\lambda_{1i}(L_i)f_{1i}(L_i)-\lambda_{2i}(L_i)f_{2i}(L_i)\nonumber\\
=&\left(\alpha_i \varphi_{2i}^2\left(\frac{\varphi_i^2-\eta_i^2}{\eta_i}\right)\right)(L_i)>0.
		\end{align}
And we emphasize that all of $Z_{i_l}(0)$ $(l=2,\cdots,\deg(J^i_M))$ being negative is a sufficient condition for the diagonal matrix $\widetilde{M}_{ik}$ being positive definite. Indeed, just like for the simple star-shaped case, it is easy to check from \eqref{trthe} that $\widetilde{\theta}_{i_2}$ is positive under $Z_{i_l}(0)<0$. 	
To get the positiveness of the matrix $\widetilde M_i$, one needs to check from \eqref{treetheta2zeta} that $\theta_{i2}$ and $\zeta_i$ are all positive under $Z_{i_l}(0)<0$, which is indeed the case by noticing \eqref{trzp}.
On the other side, from \eqref{def02} and \eqref{defZ&W}, one has
\begin{equation}\label{treeequiv}
Z_{i_l}(0)<0\Longleftrightarrow  \eta_{i_l}(0)>1.
\end{equation}
This, together with the equations \eqref{eta0j}, satisfied by the $\eta_{j}$ ($j=2,\cdots,n$) imply that each $\widetilde M_i$ is positive definite.
As a consequence, $B(t)$ is positive definite. \\
Extending these computations to the nonlinear system can be done exactly as in the star-shaped case (see Section \ref{sec:nonlinear}). We omit the details here.

 \section{Conclusion \textcolor{black}{and perspectives}}\label{con}
 In this paper, our main contribution is to investigate the exponential stability of star-shaped and tree-shaped networks of flows modeled by Saint-Venant equations, where all channels are in the subcritical flow regime. The network consists one main channel that separates into 
 several
 tributaries (or subchannels) that can separate themselves into several tributaries and so on, forming multi-node junctions inside the network. In particular, we demonstrate that stability can be achieved without applying any control at the internal junctions but instead by imposing controls only at the downstream ends of the network.\\
 
To establish this result, we construct a Lyapunov function with functional coefficients and rigorously verify the stability of this nonlinear system in the $H^2$ norm. One interesting perspective that remains open would be to deal with the incorporation of arbitrary slopes. While we expect that cases where the slope is smaller than the friction can still be handled with our approach, the cases where the influence of slope is larger than the influence of the friction have a completely different behaviors, and in particular would introduce absolute values in the calculation of $\eta_i$, 
significantly impacting the subsequent analyses. In fact, it would be surprising that in these cases if similar theorems held for arbitrarily large length of the domain. This presents a promising direction for future research.\\

\textcolor{black}{While our results deal with any tree-shaped network, another promising direction of research is the stabilization of networks that contains a cycle, as it can be found in \cite{gugat2025stabilization,Lizhuang2019}. For such cyclic graphs, new difficulties occur and, in particular, feedback laws that work for tree-shaped graphs sometimes do not ensure stabilization on networks with cycles (see \cite{gugat2025stabilization}). In \cite{gugat2025stabilization}, a feedback control that allows to stabilize a system with a single cycle was proposed and it was showed through a careful analysis that there are intrinsic limitations when one cannot act on internal nodes. As such, this is the main remaining question for the stabilization of general networks of flows with cycles modeled by Saint-Venant equations.}

 \section*{Acknowledgments}
 The authors would like to thank the National Natural Science Foundation of China (No. 12171368) and Fundamental Research Funds for the Central Universities, China (2023-3-ZD-04) as well as China Scholarship Council (No. 202206260108). The authors would also like to thank the Young Talent France-China 2024 program of the French Embassy and the ANR-Tremplin StarPDE ANR-24-ERCS-0010.

\appendix

\section{Computations for Lemma \ref{lem1}}\label{app:A}
In this section, we give technical but tedious computations used in the proof of Lemma \ref{lem1}.\\

\textbf{Computation of $\lambda_{2i}\left(2\gamma_{2i}-\gamma_{1i} \right)-\lambda_{1i}\delta_{2i}$}

We can obtain from the expressions of $\gamma_{1i}$, $\delta_{1i}$, $\gamma_{2i}$ and $\delta_{2i}$ in \eqref{mall>0} and \eqref{lam1} that
\begin{equation}
\begin{aligned}
    &\lambda_{2i}\left(2\gamma_{2i}-\gamma_{1i} \right)-\lambda_{1i}\delta_{2i}\\
    =&-\frac{(H^*_{i})_x}{H^*_i}\lambda_{1i}\lambda_{2i}\left(\frac{1}{2}+\frac{\lambda_{2i}}{V^*_i}-\frac{\textcolor{black}{p}\lambda_{2i}}{2\sqrt{gH^*_i}}+\frac{3\lambda_{2i}}{4\lambda_{1i}}-\frac{3\lambda_{1i}}{4\lambda_{2i}}-\frac{\lambda_{1i}}{V^*_i}-\frac{\textcolor{black}{p}\lambda_{1i}}{2\sqrt{gH^*_i}}\right)\\
    =&-\frac{(H^*_{i})_x}{H^*_i}\lambda_{1i}\lambda_{2i}\left(\frac{1}{2}+\frac{\lambda_{2i}-\lambda_{1i}}{V^*_i}-\frac{\textcolor{black}{p}(\lambda_{1i}+\lambda_{2i})}{2\sqrt{gH^*_i}}+\frac{3\lambda_{2i}}{4\lambda_{1i}}-\frac{3\lambda_{1i}}{4\lambda_{2i}}\right)\\
    =&-\frac{(H^*_{i})_x}{H^*_i}\lambda_{1i}\lambda_{2i}\left(-\frac{3+2\textcolor{black}{p}}{2}-\frac{3\sqrt{gH_i^*}V_i^*}{\lambda_{1i}\lambda_{2i}}\right)\\
    =&\left( H_i^*\right)_x \left[ \frac{3+2\textcolor{black}{p}}{2}\frac{gH_i^*-V_i^{*2}}{H_i^*}+\frac{3\sqrt{gH_i^*}V_i^*}{H_i^*}\right].
\end{aligned}
\end{equation}
\textbf{{Computation of $\displaystyle\frac{\gamma_{2i}}{\lambda_{2i}}$}}

From \eqref{mall>0} and using $H^*_iV^*_i=Q_i$, we have 
\begin{equation}\label{comb22}
\begin{aligned}
   \frac{\gamma_{2i}}{\lambda_{2i}}=&-\frac{(H_i^*)_x}{H_i^*}\lambda_{1i}\left[\frac{1}{4\lambda_{2i}}+\frac{1}{V_i^*}-\frac{\textcolor{black}{p}}{2\sqrt{gH_i^*}}\right]\\
   =&(H_i^*)_x\left[\frac{-\lambda_{1i}}{4H_i^*\lambda_{2i}}-\frac{\lambda_{1i}}{H_i^*V_i^*}+\frac{\textcolor{black}{p}\lambda_{1i}}{2H_i^*\sqrt{gH_i^*}}\right]\\
   =&(H_i^*)_x\left[\frac{\lambda_{2i}-\lambda_{2i}-\lambda_{1i}}{4H_i^*\lambda_{2i}}-\frac{\lambda_{1i}}{H_i^*V_i^*}+\frac{\textcolor{black}{p}\lambda_{1i}}{2H_i^*\sqrt{gH_i^*}}\right]\\
   =&(H_i^*)_x\left[\frac{1}{4H_i^*}-\frac{\sqrt{gH^*_i}}{2H_i^*(\sqrt{gH^*_i}-V^*_i)}-\frac{\sqrt{gH^*_i}+V^*_i}{H_i^*V_i^*}+\frac{\textcolor{black}{p}(\sqrt{gH^*_i}+V^*_i)}{2H_i^*\sqrt{gH_i^*}}\right]\\
   =&(H_i^*)_x\left[\frac{1}{4H_i^*}-\frac{\sqrt{gH^*_i}}{2H_i^*(\sqrt{gH^*_i}-V^*_i)}-\frac{\sqrt{gH^*_i}}{H_i^*V_i^*}-\frac{1}{H_i^*}+\frac{\textcolor{black}{p}}{2H_i^*}+\frac{\textcolor{black}{p}V^*_i}{2H_i^*\sqrt{gH_i^*}}\right]\\
   =&(H_i^*)_x\left[\frac{2\textcolor{black}{p}-3}{4}\frac{1}{H_i^*}-\frac{\sqrt{gH^*_i}}{2H_i^*(\sqrt{gH^*_i}-V^*_i)}-\frac{\sqrt{gH^*_i}}{Q_i}+\frac{\textcolor{black}{p}Q_i}{2H_i^{*2}\sqrt{gH_i^*}}\right].
   \end{aligned}
\end{equation}
\section{Proof of Remark \ref{remf}}
\label{app:nontrivial}
From \eqref{tem} and \eqref{comb1}, we can obtain that
\begin{align}\label{tem1}
   e^{\int_{0}^{x}2\frac{\gamma_{2i}}{\lambda_{1i}}d\xi}  = &\varphi_i(x)e^{\int_{0}^{x}\left(\frac{2\gamma_{2i}-\gamma_{1i}}{\lambda_{1i}}-\frac{\delta_{2i}}{\lambda_{2i}}\right)d\xi}\nonumber\\
   =&\varphi_i(x)\left( \frac{H_i^*(x)}{H_i^*(0)}\right)^{\frac{\textcolor{black}{3+2p}}{2}}\cdot \frac{\lambda_{2i}(x)}{\lambda_{2i}(0)} \cdot \frac{\lambda_{1i}(0)}{\lambda_{1i}(x)}.
\end{align}
Moreover, from \eqref{comb1} and \eqref{mall>0}, we can compute
\begin{equation}\label{uus1}
\begin{aligned}
    f(s):&=e^{\int_{0}^{s}\left(\frac{2\gamma_{2i}-\gamma_{1i}}{\lambda_{1i}}-\frac{\delta_{2i}}{\lambda_{2i}}\right)d\varepsilon}\frac{\gamma_{2i}(s)}{\lambda_{2i}(s)}
     =\left( \frac{H_i^*(s)}{H_i^*(0)}\right)^{\frac{\textcolor{black}{3+2p}}{2}}\cdot \frac{\lambda_{1i}(0)}{\lambda_{2i}(0)} \cdot \frac{\gamma_{2i}(s)}{\lambda_{1i}(s)}\\
     &=\left( \frac{H_i^*(s)}{H_i^*(0)}\right)^{\frac{\textcolor{black}{3+2p}}{2}}\cdot \frac{\lambda_{1i}(0)}{\lambda_{2i}(0)} \cdot\frac{-(H_i^*)_s}{H_i^*(s)}\lambda_{2i}(s)\left[\frac{1}{4\lambda_{2i}(s)}+\frac{1}{V_i^*(s)}-\frac{\textcolor{black}{p}}{2\sqrt{gH_i^*(s)}}\right]\\
     &=-\left(\frac{1}{H_i^*(0)}\right)^{\frac{\textcolor{black}{3+2p}}{2}}\cdot\frac{\lambda_{1i}(0)}{\lambda_{2i}(0)}\cdot\left( H_i^*(s)\right)^{\frac{\textcolor{black}{1+2p}}{2}} \cdot(H_i^*)_s\cdot\left[ -\frac{\textcolor{black}{3+2p}}{4}+\frac{\sqrt{gH_i^*(s)}}{V^*_i(s)}+\frac{\textcolor{black}{p}V^*_i(s)}{2\sqrt{gH_i^*(s)}}\right]\\
     &=-\left(\frac{1}{H_i^*(0)}\right)^{\frac{\textcolor{black}{3+2p}}{2}}\cdot\frac{\lambda_{1i}(0)}{\lambda_{2i}(0)}\cdot\left( H_i^*(s)\right)^{\frac{\textcolor{black}{1+2p}}{2}} \cdot(H_i^*)_s\cdot\left[ -\frac{\textcolor{black}{3+2p}}{4}+\frac{\sqrt{g}(H_i^*(s))^{\frac{3}{2}}}{Q_i}+\frac{\textcolor{black}{p}Q_i}{2\sqrt{g}(H_i^*(s))^{\frac{3}{2}}}\right]\\
     &=\left(\frac{1}{H_i^*(0)}\right)^{\frac{\textcolor{black}{3+2p}}{2}}\cdot\frac{\lambda_{1i}(0)}{\lambda_{2i}(0)}\cdot\left[ \frac{\textcolor{black}{3+2p}}{4}\left( H_i^*(s)\right)^{\frac{\textcolor{black}{1+2p}}{2}} (H_i^*)_s-\frac{\sqrt{g}(H_i^*(s))^{\textcolor{black}{2+p}}(H_i^*)_s}{Q_i}-\frac{\textcolor{black}{p}Q_i(H_i^*(s))^{\textcolor{black}{p-1}}(H_i^*)_s}{2\sqrt{g}}\right].
    \end{aligned}
\end{equation}
As a consequence, for any $x\in[0,x_i^0)$, we have
\begin{equation}\label{uus2}
\begin{aligned}
&\int_0^{x}f(s)\, ds\\
=&\left(\frac{1}{H_i^*(0)}\right)^{\frac{\textcolor{black}{3+2p}}{2}}\cdot\frac{\lambda_{1i}(0)}{\lambda_{2i}(0)}\int_0^x\left[ \frac{\textcolor{black}{3+2p}}{4}\left( H_i^*(s)\right)^{\frac{\textcolor{black}{1+2p}}{2}} (H_i^*)_s-\frac{\sqrt{g}(H_i^*(s))^{\textcolor{black}{2+p}}(H_i^*)_s}{Q_i}-\frac{\textcolor{black}{p}Q_i(H_i^*(s))^{\textcolor{black}{p-1}}(H_i^*)_s}{2\sqrt{g}}\right]\, ds\\
=&\left(\frac{1}{H_i^*(0)}\right)^{\frac{\textcolor{black}{3+2p}}{2}}\cdot\frac{\lambda_{1i}(0)}{\lambda_{2i}(0)}\left[\frac{1}{2}\Delta_{\frac{\textcolor{black}{3+2p}}{2}} -\frac{\sqrt{g}}{\textcolor{black}{(3+p)}Q_i}\Delta_{\textcolor{black}{3+p}} -\frac{Q_i}{2\sqrt{g}}\Delta_{\textcolor{black}{p}}\right],
\end{aligned}
\end{equation}
where, to simplify the notation, we denote:
\textcolor{black}{\begin{equation}
\Delta_{\sigma}:=(H_i^*(x))^{\sigma} - (H_i^*(0))^{\sigma},  \ \sigma\in \mathbb{R}.
\end{equation}}

Substituting \eqref{tem1} and \eqref{uus2} into \eqref{exbareta}, we get
\begin{align}\label{exbaretas}
\bar{\eta}_i(x)=&\frac{\lambda_{2i}(x)}{\lambda_{1i}(x)}\varphi_i(x)+\frac{\varphi_i(x)\left( \frac{H_i^*(x)}{H_i^*(0)}\right)^{\frac{\textcolor{black}{3+2p}}{2}}\cdot \frac{\lambda_{2i}(x)}{\lambda_{2i}(0)} \cdot \frac{\lambda_{1i}(0)}{\lambda_{1i}(x)}}{\left(\frac{1}{H_i^*(0)}\right)^{\frac{\textcolor{black}{3+2p}}{2}}\cdot\frac{\lambda_{1i}(0)}{\lambda_{2i}(0)}\left[-\frac{1}{2}\Delta_{\frac{\textcolor{black}{3+2p}}{2}} +\frac{\sqrt{g}}{(\textcolor{black}{3+p})Q_i}\Delta_{\textcolor{black}{3+p}}+\frac{Q_i}{2\sqrt{g}}\Delta_{\textcolor{black}{p}}\right]+\frac{\lambda_{1i}(0)}{\lambda_{1i}(0)-\lambda_{2i}(0)}}\nonumber\\
=&\frac{\lambda_{2i}(x)}{\lambda_{1i}(x)}\varphi_i(x)+\frac{\varphi_i(x)\left(H_i^*(x)\right)^{\frac{\textcolor{black}{3+2p}}{2}}\cdot \frac{\lambda_{2i}(x)}{\lambda_{1i}(x)} }{-\frac{1}{2}\Delta_{\frac{\textcolor{black}{3+2p}}{2}} +\frac{\sqrt{g}}{\textcolor{black}{(3+p)}Q_i}\Delta_{\textcolor{black}{3+p}} +\frac{Q_i}{2\sqrt{g}}\Delta_{\textcolor{black}{p}} +\frac{\lambda_{2i}(0)\left(H_i^*(0)\right)^{\frac{\textcolor{black}{3+2p}}{2}}}{\lambda_{1i}(0)-\lambda_{2i}(0)}}\nonumber\\
=&\frac{\lambda_{2i}(x)}{\lambda_{1i}(x)}\varphi_i(x)\left[1+\frac{\left(H_i^*(x)\right)^{\frac{\textcolor{black}{3+2p}}{2}} }{-\frac{1}{2}\Delta_{\frac{\textcolor{black}{3+2p}}{2}} +\frac{\sqrt{g}}{\textcolor{black}{(3+p)}Q_i}\Delta_{\textcolor{black}{3+p}} +\frac{Q_i}{2\sqrt{g}}\Delta_{\textcolor{black}{p}} +\frac{\lambda_{2i}(0)\left(H_i^*(0)\right)^{\frac{\textcolor{black}{3+2p}}{2}}}{\lambda_{1i}(0)-\lambda_{2i}(0)}}\right]\nonumber\\
=&\frac{\lambda_{2i}(x)}{\lambda_{1i}(x)}\varphi_i(x)\!\!\left[1\!+\!\frac{\left(H_i^*(x)\right)^{\frac{\textcolor{black}{3+2p}}{2}} }{-\frac{1}{2}\Delta_{\frac{\textcolor{black}{3+2p}}{2}} \!\!+\!\!\frac{\sqrt{g}}{\textcolor{black}{(3+p)}Q_i}\Delta_{\textcolor{black}{3+p}}\!+\!\frac{Q_i}{2\sqrt{g}}\Delta_{\textcolor{black}{p}} \!+\!\frac{\sqrt{g}}{2Q_i}(H^*_i(0))^{\textcolor{black}{3+p}}\!-\!\!\frac{1}{2}(H^*_i(0))^{\frac{\textcolor{black}{3+2p}}{2}}}\!\right]\nonumber\\
=&\frac{\lambda_{2i}(x)}{\lambda_{1i}(x)}\varphi_i(x)\frac{\frac{1}{2}(H_i^{*}(x))^{\frac{\textcolor{black}{3+2p}}{2}}+\frac{\sqrt{g}}{\textcolor{black}{(3+p)}Q_i}(H_i^{*}(x))^{\textcolor{black}{3+p}}+\frac{\textcolor{black}{(1+p)}\sqrt{g}}{\textcolor{black}{2(3+p)}Q_i}(H_i^{*}(0))^{\textcolor{black}{3+p}}+\frac{Q_i}{2\sqrt{g}}((H^*_i(x))^{\textcolor{black}{p}}-(H^*_i(0))^{\textcolor{black}{p}})}{-\frac{1}{2}(H_i^{*}(x))^{\frac{\textcolor{black}{3+2p}}{2}}+\frac{\sqrt{g}}{\textcolor{black}{(3+p)}Q_i}(H_i^{*}(x))^{\textcolor{black}{3+p}}+\frac{\textcolor{black}{(1+p)}\sqrt{g}}{\textcolor{black}{2(3+p)}Q_i}(H_i^{*}(0))^{\textcolor{black}{3+p}}+\frac{Q_i}{2\sqrt{g}}((H^*_i(x))^{\textcolor{black}{p}}-(H^*_i(0))^{\textcolor{black}{p}})}\nonumber\\
=&m_i(x)\frac{\lambda_{2i}(x)}{\lambda_{1i}(x)}\varphi_i(x).
\end{align}
which completes the proof of Remark \ref{remf}.

\section{Proof of Theorem \ref{thm:single}}
\label{app:single}
The proof follows from the proof of Theorem \ref{thm0}. We only look here at the linearized system, the nonlinear system can be treated exactly as in Section \ref{sec:nonlinear} (or as in \cite{HS} given that we have a single channel). Let us set
\begin{equation}
\label{eq:defkk0}
k = \mathcal{G}_{2}'(H^{*}(L)),\;\; k_{0}=\mathcal{G}_{1}'(H^{*}(0)).
\end{equation}
Using the change of variables \eqref{RiemannIn}, the system becomes \eqref{chaform} 
where we drop the index $i$
and the boundary conditions are 
\begin{equation}
    y_{1}(t,L) = c y_{2}(t,L),\;\; y_{1}(t,0) = c_{0} y_{2}(t,0),
\end{equation}
where $c$ is given by \eqref{eqcj} (with 
$H^{*}(L)$ and $k$ instead of $H_{j}^{*}(L_j)$ and $k_{j}$) and $c_{0}$ is given by
\begin{equation}
\label{eq:defc0}
    c_{0} = \frac{k_{0}H^{*}(0)+\sqrt{gH^{*}(0)}}{k_{0}H^{*}(0)-\sqrt{gH^{*}(0)}}.
\end{equation}
Now, we use the Lyapunov function 
\begin{equation}
    V(t) = \int_{0}^{L}\left(f_{1}(x)y_{1}^{2}(t,x)+f_{2}(x)y_{2}^{2}(t,x)\right)dx,
\end{equation}
with $f_{1}$ and $f_{2}$ given by 
\begin{equation}\label{def022}
				f_{1}(x)=\alpha \phi_{1}^2\frac{1}{\lambda_{1}\eta}, \quad f_{2}(x)=\alpha \phi_{2}^2\frac{\eta}{\lambda_{2}}, 
		\end{equation}
        where $\eta$ is the solution to
        \begin{equation}\label{eta0j2}
			\left\{
			\begin{aligned}
				&\eta'=\left| \frac{\delta_{1}\phi}{\lambda_{1}}+\frac{\gamma_{2}}{\lambda_{2}\phi}\eta^2 \right|+\varepsilon, \\
				&\eta(0)=1+\varepsilon,
			\end{aligned}
			\right.
		\end{equation}
        with $\phi=\phi_1/\phi_2$ and 
        \begin{equation}\label{varphi122}
\phi_{1}(x)=e^{\int_{0}^{x}\frac{\gamma_{1}(s)}{\lambda_{1}(s)}ds},\quad \phi_{2}(x)=e^{-\int_{0}^{x}\frac{\delta_{2}(s)}{\lambda_{2}(s)}ds}.
\end{equation}
where $\gamma_{1}$,$\gamma_{2}$, $\delta_{1}$, $\delta_{2}$, $\lambda_{1}$, $\lambda_{2}$ are defined by \eqref{mall>0} and \eqref{lam1} by dropping the index $i$.
As previously (see \eqref{eq:computelyap}--\eqref{B(t)}) 
\begin{equation}
    \dot V(t) = -B(t) - \int_{0}^{L} (y_{1},y_{2})N(x)(y_{1},y_{2})^{T}dx,
\end{equation}
where $N(x)$ is defined as \eqref{inner01} by dropping the index $i$ and thus is still positive definite but, for a single channel, $B(t)$ is much simpler than in \eqref{B(t)} and simply writes
\begin{equation}
\label{eq:defBsingle}
    B(t) = -y_{2}^{2}(t,0)(f_{1}(0)\lambda_{1}(0)c_{0}^{2}-f_{2}(0)\lambda_{2}(0)) + y_{2}^{2}(t,L)(f_{1}(L)\lambda_{1}(L)c^{2}-f_{2}(L)\lambda_{2}(L)).
\end{equation}
The second term is identical to the case of the star-shaped system and is positive thanks to the second condition of \eqref{eq:condsingle}, \eqref{eq:defkk0}, \eqref{eta-phi03}--\eqref{def002}. The only difference comes from the first term, but, using the definition of $f_{1}$ and $f_{2}$, together with \eqref{varphi12} we have
\begin{equation}
\label{eq:term1single}
f_{1}(0)\lambda_{1}(0)c_{0}^{2}-f_{2}(0)\lambda_{2}(0) = \alpha\left(\frac{c_{0}^{2}}{\eta(0)}-\eta(0)\right) =\alpha\left(\frac{c_{0}^{2}}{1+\varepsilon}-(1+\varepsilon)\right)  , 
\end{equation}
From \eqref{eq:condsingle}, \eqref{eq:defkk0} and \eqref{eq:defc0}, we deduce that $c_{0}^{2}\leq 1$ and therefore from \eqref{eq:term1single} and \eqref{eq:defBsingle}, since $\alpha>0$, $B(t)$ is positive definite with respect to $y_{2}^{2}(t,0)$ and  $y_{2}^{2}(t,L)$ and this concludes the proof.

  \bibliographystyle{plain}
	\bibliography{SV-networks}
	
\end{document}